\documentclass[12pt]{article}
\usepackage{rotating}
\usepackage{amssymb}
\usepackage{color}
\usepackage{amsmath,amsthm}
\usepackage{stmaryrd}
\usepackage{amsfonts}
\usepackage{graphicx}
\usepackage{graphics}
\usepackage{listings}
\usepackage{pgf,tikz}
\usepackage{mathrsfs}
\usepackage{scalerel,stackengine}
\usepackage{booktabs}
\usepackage{siunitx}
\usepackage{float}
\usepackage{subfigure}
\usepackage{multirow}
\usepackage{enumitem}
\usepackage{multicol}
\stackMath
\newcommand\reallywidehat[1]{%
\savestack{\tmpbox}{\stretchto{%
\scaleto{%
\scalerel*[\widthof{\ensuremath{#1}}]{\kern-.6pt\bigwedge\kern-.6pt}%
{\rule[-\textheight/2]{1ex}{\textheight}}
}{\textheight}%
}{0.5ex}}%
\stackon[1pt]{#1}{\tmpbox}%
}
\usepackage[pagebackref=false,colorlinks,linkcolor=blue,filecolor=green,urlcolor=blue,citecolor=magenta]{hyperref}
\usetikzlibrary{arrows}
\newcommand\qan{{\quad\hbox{and}\quad}}

\newcommand\qin{{\quad\hbox{in}\quad}}
\newcommand\qon{{\quad\hbox{on}\quad}}
\renewcommand\div{{\mathrm{div}}}
\newcommand\bdiv{{\mathbf{div}}}
\newcommand\bta{{\boldsymbol\tau}}
\newcommand\bze{{\boldsymbol\zeta}}

\newcommand\tr{{\mathrm{tr}}}
\newcommand\bsi{{\boldsymbol\sigma}}

\newcommand\btet{{\boldsymbol\theta}}
\newcommand\bn{{\mathbf n}}
\newcommand\bu{{\mathbf u}}
\newcommand\bz{{\mathbf z}}
\newcommand\by{{\mathbf y}}
\newcommand\bv{{\mathbf v}}

\newcommand\bw{{\mathbf w}}
\newcommand\disp{\displaystyle}

\newcommand{\Pcalbf}{{\boldsymbol{\mathcal P}}}
\newcommand{\Pcalbb}{\mathcal{P}\hspace{-1.5ex}\mathcal{P}}
\numberwithin{equation}{section}
\newtheorem{Theorem}{Theorem}[section]

\newtheorem{Definition}[Theorem]{Definition}

\newtheorem{prob}{Problem}
\newtheorem{lemma}[Theorem]{Lemma}

\begin{document}
\date{\small\textsl{\today}}
\title{\textsl{Analysis of weak Galerkin mixed FEM based on the velocity--pseudostress formulation for Navier--Stokes equation on polygonal meshes}
}
\author{
	{\sc Zeinab Gharibi}\thanks{CI$^2$MA and GIMNAP-Departamento de Matem\'atica, Universidad del B\'io-B\'io, Casilla 5-C, Concepci\'on, Chile, email: {\tt zgharibi@ubiobio.cl}.} 
	\quad
	{\sc Mehdi Dehghan}\thanks{Department of Applied Mathematics, Faculty of Mathematics and 
		Computer Sciences, Amirkabir University of Technology (Tehran Polytechnic), No. 424, Hafez Ave., 
		Tehran, 15914, Iran, email: {\tt mdehghan@aut.ac.ir}.}		
}
\maketitle
\vspace{-1cm} \noindent
\noindent\linethickness{.5mm}{\line(1,0){500}}
\begin{abstract}
The present article introduces, mathematically analyzes, and numerically validates a new weak Galerkin (WG) mixed-FEM based on Banach spaces for the stationary Navier--Stokes equation in pseudostress-velocity formulation. More precisely, a modified pseudostress tensor, called $ \bsi $, depending on the pressure, and the diffusive and convective terms has been introduced in the proposed technique, and a dual-mixed variational formulation has been derived where the aforementioned pseudostress tensor and the velocity, are the main unknowns of the system, whereas the pressure is computed via a post-processing formula. Thus, it is sufficient to provide a WG space for the tensor variable and a space of piecewise polynomial vectors of total degree at most 'k' for the velocity.
Moreover, in order
to define the weak discrete bilinear form, whose continuous version involves the classical divergence operator, the weak divergence operator as a well-known alternative for the classical divergence operator in a suitable discrete subspace is proposed.
 The well-posedness of the numerical solution is proven using a fixed-point approach and the discrete versions of the Babuška--Brezzi theory and the Banach--Nečas--Babuška theorem. Additionally, an a priori error estimate is derived for the proposed method. Finally, several numerical results illustrating the method's good performance and confirming the theoretical rates of convergence are presented.\\

Error analysis.		\vspace{.5cm}
\end{abstract}
\noindent
\noindent\linethickness{.5mm}{\line(1,0){500}}

\section{Introduction}
\subsection{Scope}
The weak Galerkin method (WG), introduced in \cite{Wang13} for second-order elliptic equations, approximates the differential operators in the variational formulation through a framework that emulates the theory of distributions for piecewise polynomials. Additionally, the usual regularity of the approximating functions in this method is compensated by carefully-designed stabilizers. In turn, the WG methods have been studied for solving various models (cf. (cf.  \cite{Mu14,Zhai16,Zhai16,Mu20,Wang13,Wang14,Chen16,Liu19,L.Mu15,
	Dehghan21,Gharibi21,Dehghan22})), demonstrating their substantial potential as a formidable numerical method in scientific computing.
The key distinction between WG methods and other existing finite element techniques lies in their utilization of weak derivatives and weak continuities in formulating numerical schemes based on conventional weak forms for the underlying PDE problems.
Because of their significant structural pliability, WG methods are aptly suited for a broad class of PDEs, ensuring the requisite stability and accuracy in approximations.

\medskip\noindent

On the other hand, the development of certain mixed finite element methods has become a new research field aimed at solving linear and nonlinear problems using velocity-pressure, stress-velocity-pressure, and pseudostress-velocity-based formulations. A principal benefit of stress-based and pseudostress-based formulations is the immediate calculation of physical quantities, such as stress, without the need for velocity derivatives, preventing the degradation of precision. Despite the extensive research on the weak Galerkin FEM, there are few publications on the weak Galerkin mixed-FEM in the literature.
 Specifically, and up to our knowledge, \cite{Wang14} and \cite{Gharibi21} are the only two works in which the weak Galerkin mixed-FEM for linear and nonlinear PDEs based on the dual formulation are proposed and analyzed, including corresponding optimal errors estimates.
 Although many research works have been studied and analyzed on the weak Galerkin method for the Navier--Stokes problem (cf. \cite{Liu19,Zhang18,Zhang19,Zhang20,Hu19}), most existing schemes are extracted by the Stokes-type (or primal-mixed) formulation, where the velocity and pressure are principle unknowns.

\medskip\noindent

According to the above discussion, the objective of the present paper is to develop and design a new mixed FEM based on the nonstandard Banach spaces drawn by Gatica et al. \cite{Gatica21} within the framework of the weak Galerkin method to solve the stationary Navier--Stokes equation.
 More precisely, we are particularly interested in the development of mixed formulations not involving any augmentation procedure (as done, e.g. in \cite{Camano16,Camano17}). To this end, we now extend the applicability of the approach employed in \cite{Camano21} for the fluid model, in the framework weak Galerkin method. 
In fact, instead of using the primal or dual-mixed method (see e.g. \cite{Wang14,Gharibi21}), we now employ a modified mixed formulation and  adapt the approach from \cite{Gharibi24} to propose, up to our knowledge by the first time, a weak Galerkin mixed-FEM for Navier--Stokes, which consists of introducing the gradient of velocity and a vector version of the Bernoulli tensor as further unknowns. In this way, and besides eliminating the pressure, which can be approximated later on via postprocessing, the resulting mixed variational formulation does not need to incorporate any augmented term, and it yields basically the same Banach saddle-point structure.

\subsection{Outline and notations} 
The remainder of the paper is organized as follows. In Sec. \ref{sec2}, we refer to Ref. \cite{Camano21} to introduce the model of interest, recapitulate the associated dual-mixed variational formulation with the unknowns $\bsi$ and $\bu$ living in suitable Banach spaces, and state the main result establishing its well-posedness. We then introduce the WG discretization in Sec. \ref{sec3}, following Refs. \cite{Wang14} and \cite{Gharibi21}. This section comprises four main parts: first, we state basic assumptions on the mesh; second, we define the local WG space, projections, and weak differential operators; third, we discuss their approximation properties; and fourth, we derive the global WG subspace and the disceret scheme.
In Sec. \ref{sec4}, we analyze the solvability of our discrete scheme using a fixed-point strategy. To accomplish this, we derive common estimates for the bilinear and trilinear forms, as well as the discrete inf-sup condition. The classical Banach fixed-point theorem then allows us to obtain the principal result.
Sec. \ref{sec5} is dedicated to deriving a priori error estimates for the numerical solution under a small data assumption. Finally, in Sec. \ref{sec6}, we present some numerical experiments to confirm the theoretical correctness and effectiveness of the discrete schemes.

For any vector fields $ \mathbf{v}=(v_1, v_2)^{\tt t} $ and $ \mathbf{w}=(w_1, w_2)^{\tt t} $, we set the gradient, divergence and tensor product operators as
\[
\boldsymbol{\nabla}\mathbf{v}:=(\nabla v_1, \nabla v_2), \quad \operatorname{div}(\mathbf{v}):=\partial_x v_1 + \partial_y v_2, 
\qan \mathbf{v}\otimes\mathbf{w}\,:=\, (v_i \, w_j)_{i,j=1,2} \,,
\]

respectively. In addition, denoting by $\mathbb{I}$ the identity matrix of $\mathrm{R}^2$, for any tensor fields 
$\bta=(\tau_{ij}),~\bze=(\zeta_{ij})\in \mathrm{R}^{2\times 2}$, we write as usual
\[
\bta^{\mathtt{t}}:=(\tau_{ji}), \quad \operatorname{tr}(\bta):=\tau_{11}+\tau_{22},\quad 
\bta^{\mathtt{d}}:= \bta-\dfrac{1}{2}\operatorname{tr}(\bta)\mathbb{I}\,,
\qan \bta: \bze:= \sum_{i,j=1}^{2}\tau_{ij}\zeta_{ij} \,,
\]
which corresponds, respectively, to the transpose, the trace, and the deviator tensor of $ \bta $, and to the
tensorial product between $ \bta $ and $ \bze $.

Throughout the paper, given a bounded domain $\Omega$, we let $\mathcal{O}$ be any given open subset of $\Omega$. By $(\cdot, \cdot)_{0,\mathcal{O}}$ and $\Vert\cdot\Vert_{0,\mathcal{O}}$ we denote the usual integral inner product and the corresponding norm of $\mathrm{L}^{2}(\mathcal{O})$, respectively. For positive integers $m$ and $r$, we shall use the common notation for the Sobolev spaces $\mathrm{W}^{m,r}(\mathcal{O})$ with the corresponding norm and semi-norm $\Vert\cdot\Vert_{m,r,\mathcal{O}}$ and $\vert\cdot\vert_{m,r,\mathcal{O}}$, respectively; and if $r=2$, we set $\mathrm{H}^{m}(\mathcal{O}):=\mathrm{W}^{m,2}(\mathcal{O})$, $\Vert\cdot\Vert_{m,\mathcal{O}}:=\Vert\cdot\Vert_{m,2,\mathcal{O}}$ and $\vert\cdot\vert_{m,\mathcal{O}}:=\vert\cdot\vert_{m,2,\mathcal{O}}$. 
If $\mathcal{O}=\Omega$, the subscript will be omitted. Furthermore, $\textbf{M}$ and $\mathbb{M}$ represent corresponding vectorial and tensorial counterparts of the scalar functional space $\mathrm{M}$.
On the other hand, given $t\in (1,+\infty)$, letting $ \operatorname{\mathbf{div}} $ 
be the usual divergence operator $ \operatorname{div} $ acting along the rows of a given tensor, we introduce the standard Banach space
\[
\mathbb{H}(\bdiv_t;\Omega)\,:=\,\Big\{\bta\in \mathbb{L}^{2}(\Omega):\quad  
\bdiv(\bta)\in \mathbf{L}^{t}(\Omega)\Big\}\,,
\]
equipped with the usual norm
\[
\|\bta\|_{\bdiv_t ; \Omega}\,:=\,\|\bta\|_{0,\Omega} +
\|\bdiv(\bta)\|_{0,t; \Omega}, \quad\quad 
\forall\,\bta\in \mathbb{H}(\bdiv_t ; \Omega),
\]
\section{The model problem and its continuous formulation}\label{sec2}
Consider a spatial bounded domain  $\Omega\subset\mathbb{R}^{d}$ ($d=2,3$)  with a Lipschitz continuous boundary $\partial\Omega$ with outward-pointing unit normal $\mathbf{n}$. 
We focus on solving the Navier--Stokes equation with viscosity $\nu$, where, given the body force term $\textbf{f}\in \textbf{L}^{2}(\Omega)$ and suitable boundary data $\mathbf{g}\in \textbf{H}^{1/2}(\partial\Omega)$, the objective is to find a velocity field $\mathbf{u}$ and a pressure field $p$ such that
\begin{equation}\label{Eq0_1}
\begin{array}{rcll}
-\nu\,\boldsymbol{\Delta} \bu+\bu\cdot\boldsymbol{\nabla} \bu+\nabla p \,&=&\,\textbf{f}  &\qin\Omega\,, \\[2ex]
\div(\bu)\,&=&\,0  & \qin \Omega\,, \\[2ex]
\bu \,&=&\, \mathbf{g} & \qon \partial\Omega\,,
\end{array}
\end{equation}
In addition, in order to guarantee uniqueness of the pressure, this unknown will be sought in the space
\begin{equation}\label{eq:cond.p}
	L_0^{2}(\Omega)\,:=\, \Big\{q\in L^{2}(\Omega):\quad \int_{\Omega}p \,=\, 0 \Big\}\,.
\end{equation}
Note that, due to the incompressibility of the fluid (cf. second row of \eqref{Eq0_1}), $ \mathbf{g} $ must satisfy
\[
\disp\int_{\Omega}\mathbf{g}\cdot\bn \,=\, 0 \,.
\]
For the subsequent analysis we assume that coefficient $ \nu $ is piecewise constant and positive.

Next, to obtain a velocity--pseudostress formulation, the first step is to rewrite equation \eqref{Eq0_1} so that stress and velocity are only unknowns of equation. To achieve this, we introduce a tensor field denoted by $\bsi$, represented as
\begin{equation}\label{As1}
 \boldsymbol{\sigma}\,:=\,\nu \boldsymbol{\nabla} \mathbf{u}-\mathbf{u}\otimes\mathbf{u}-(p+r_{\mathbf{u}})\mathbb{I}\quad \qin\Omega\,,
\end{equation}
where 
\[
r_{\mathbf{u}}\,:=\,-c_{r}(\operatorname{tr}(\mathbf{u}\otimes \mathbf{u}),1)_{0,\Omega}\,=\,-c_{r}(\mathbf{u},\mathbf{u})_{0,\Omega}\qquad \text{with}~~ c_{r}\,=\,\dfrac{1}{2|\Omega|}\,.
\]
In this way, applying the trace operator to both sides of \eqref{As1}, and utilizing the incompressibility condition $\operatorname{div}(\mathbf{u})=0$, one arrives at 
\begin{equation}\label{As1a}
p \,=\,-\dfrac{1}{2}\left(\operatorname{tr}(\boldsymbol{\sigma})+\operatorname{tr}(\mathbf{u}\otimes \mathbf{u})\right)-r_{\mathbf{u}}\quad \qin\Omega\,.
\end{equation}
which allows us to eliminate the pressure variable from the formulation. In turn, according
to \eqref{As1a}, the assumption \eqref{eq:cond.p} becomes
\begin{equation}\label{eq:cond.tr}
	\disp\int_{\Omega}\tr(\bsi)\,=\, 0\,.
\end{equation}
Hence, after replacing back \eqref{As1} in \eqref{Eq0_1},
gathering the resulting equation and \eqref{eq:cond.tr}, we have the following problem, which contains unknowns $\bsi$ and $\bu$.
\begin{prob}[Model problem]
\label{p:1}
Find $\boldsymbol{\sigma}, \mathbf{u}$ such that
\begin{equation*}
\begin{cases} 
\hspace{.1cm}\boldsymbol{\sigma}^{\mathtt{d}}+(\mathbf{u}\otimes\mathbf{u})^{\mathtt{d}}\,=\,\nu\,\boldsymbol{\nabla} \mathbf{u} &  \qin \Omega\,, \\[1mm]
\hspace{1.3cm}\operatorname{\mathbf{div}}( \boldsymbol{\sigma})\,=\,-\mathbf{f} &  \qin \Omega\,,\\[1mm] 
\hspace{1.1cm}\disp\int_{\Omega}\tr(\bsi)\,=\, 0\,.
\end{cases}
\end{equation*}
supplied with the following boundary condition
\[
\mathbf{u}\,=\,\mathbf{g}\quad   \qon \partial\Omega\,.
\]
\end{prob}
Next, in order to derive a velocity--pseudostress based-mixed formulation for Problem \ref{p:1}, we let $\mathbb{X}$ and $\mathbf{Y}$ be corresponding test spaces, and then proceed to
multiply the first and second equations of Problem \ref{p:1} by $\bta$ and $\bv$, respectively, and use the fact that $\tr(\bta^{\mathtt{d}})=0$, to get
\begin{equation}\label{EE1}
\dfrac{1}{\nu}\int_{\Omega}\bsi^{\mathtt{d}}: \bta^{\mathtt{d}}+\dfrac{1}{\nu}\int_{\Omega}(\bu\otimes\bu)^{\mathtt{d}}:\bta\,=\,	\int_{\Omega}
\boldsymbol{\nabla}\bu: \bta\qquad \forall\, \bta\in \mathbb{X}\,,
\end{equation}
and
\begin{equation}\label{EE2}
\int_{\Omega}\bdiv (\bsi)\cdot \bv\,=\,-\int_{\Omega}\textbf{f}\cdot \bv\qquad \forall\, \bv\in \mathbf{Y}\,,
\end{equation}
it is easy to notice that, thanks to Cauchy--Schwarz's inequality, the first term on the left hand side of \eqref{EE1} makes sense for $ \bsi,\bta\in \mathbb{L}^{2}(\Omega) $.
In turn, regarding the term on the right hand side of \eqref{EE1}, assuming originally that $ \bu\in\mathbf{H}^{1}(\Omega) $, and given $ t,t'\in (1,\infty) $, conjugate to each other, 
we can integrate by parts with $\bta \in \mathbb H(\bdiv_t;\Omega)$, so that using the Dirichlet boundary conditions provided in Problem \ref{p:1}, we obtain
\begin{equation}\label{EE1-otra}
	\int_{\Omega}
	\boldsymbol{\nabla}\bu: \bta\,=\,-\int_{\Omega}\bu\cdot\bdiv(\bta)
	+ \langle \bta \bn,\mathbf{g}\rangle \quad 
	\forall\, \bta\in \mathbb{H}(\bdiv_t;\Omega) \,,
\end{equation}
where $\langle \cdot,\cdot\rangle_{\Gamma}$ stands for the duality 
$(\mathbf H^{-1/2}(\Gamma),\mathbf H^{1/2}(\Gamma)\big)$.
Now, from the first term on the right hand side of the foregoing equation, along with the Sobolev embedding 
$ \mathbf{H}^{1}(\Omega)\subset \mathbf{L}^{t'}(\Omega) $ , we realize that it actually suffices to look for 
$\bu\in \mathbf{L}^{t'}(\Omega)$. However, it is clear from \eqref{EE1} 
that its second term is well defined if $\bu\in \mathbf{L}^{4}(\Omega)$, which yields $t'= 4$ and 
thus $t = 4/3$. 

At this point, in order to deal with the null mean value of $ \tr(\bsi) $ (cf. third row of Problem \ref{p:1}), we introduce
the subspace of $ \mathbb{H}(\bdiv_{4/3};\Omega) $ given by
\[
\mathbb{H}_0(\operatorname{\mathbf{div}}_{4/3};\Omega)
\,:=\, \Big\{\bta\in \mathbb{X}: \quad \int_\Omega \tr(\bta) \,=\, 0 \Big\}\,.
\]
Then, testing the new \eqref{EE1} against $ \bta\in\mathbb{H}(\operatorname{\mathbf{div}}_{4/3};\Omega) $ is equivalent to doing it against $ \bta\in\mathbb{H}_0(\operatorname{\mathbf{div}}_{4/3};\Omega) $,
and taking into account the above discussion, we define the testing spaces as
\[
\mathbb{X}\,:=\,\mathbb{H}_{0}(\operatorname{\textbf{div}}_{4/3};\Omega),\quad\quad \text{with}\quad\quad \Vert\cdot\Vert_{\mathbb{X}}\,:=\,\|\cdot\|_{\operatorname{\textbf{div}}_{4/3} ; \Omega}
\]
and
\[
\mathbf{Y}\,:=\,\textbf{L}^{4}(\Omega),\qquad\qquad \text{with}\quad\quad \Vert\cdot\Vert_{\mathbf{Y}}\,:=\,\Vert\cdot\Vert_{0,4,\Omega}.
\]
Let us introduce the following bilinear (and trilinear) forms 
\begin{align*}
 \mathcal{A}(\cdot,\cdot): \mathbb{X}\times\mathbb{X}\rightarrow \mathrm{R}& \qquad\qquad \mathcal{A}(\bze, \bta)\,:=\,\dfrac{1}{\nu}\int_{\Omega}\bze^{\mathtt{d}}: \bta^{\mathtt{d}}\,,\\
 \mathcal{B}(\cdot,\cdot): \mathbb{X}\times\mathbf{Y}\rightarrow \mathrm{R}& \qquad\qquad \mathcal{B}(\bta,\bv)\,:=\,\int_{\Omega}\bv\cdot\operatorname{\textbf{div}}(\bta)\,,\\
  \mathcal{C}(\cdot;\cdot,\cdot): \mathbf{Y}\times\mathbf{Y}\times\mathbb{X}\rightarrow \mathrm{R}& \qquad \qquad\mathcal{C}(\textbf{w},\bv;\bta)\,:=\,\dfrac{1}{\nu}\int_{\Omega}(\textbf{w}\otimes\bv)^{\mathtt{d}}:\bta\,,
\end{align*}
and the linear functionals (associated to given data)
\begin{equation*}
\mathcal{G}(\bta)\,:=\,\int_{\partial\Omega}\textbf{g}\cdot \bta\mathbf{n}\quad\forall\,\bta\in\mathbb{X}
\qan \mathcal{F}(\bv)\,:=\,-\int_{\Omega}\textbf{f}\cdot \bv\quad \forall\,\bv\in\mathbf{Y}\,.
\end{equation*}
Then, with these forms at hand, the variational formulation of Problem \ref{p:1} reads as follows:
\begin{prob}[variational problem]
\label{p:2}
Find the tensor $\boldsymbol{\sigma}\in \mathbb{X}$ and the velocity $\mathbf{u}\in\mathbf{Y}$ such that 
\begin{equation*}
\left\{\begin{array}{ll}
\mathcal{A}(\boldsymbol{\sigma}, \bta)+\mathcal{C}(\mathbf{u},\mathbf{u};\bta)+\mathcal{B}(\bta,\mathbf{u})\,=\,\mathcal{G}(\bta)\quad\forall\,\bta\in\mathbb{X}\,,  \\[3mm]
\hspace{4.2cm}\mathcal{B}(\boldsymbol{\sigma},\mathbf{v})\,=\,\mathcal{F}(\mathbf{v})\quad\forall\,\bv\in\mathbf{Y}\,. 
\end{array}\right.
\end{equation*}
\end{prob}
The solvability result concerning Problem \eqref{p:2} is established as follows.
\begin{Theorem}\label{WP_C}
Let $\delta>0$ be constant related to the inf-sup condition of the linear part of the left-hand side of Problem \ref{p:2} (cf. Ref. \cite[Eq. (3.29)]{Camano21}) and $c_{g}$ be the upper bound of $\mathcal{G}(\cdot)$, define the ball
\begin{equation*}
\widehat{\mathbf{Y}}\,:=\,\Big\{\mathbf{z}\in \mathbf{Y}:\quad\quad \Vert \mathbf{z}\Vert_{\mathbf{Y}}\,\leq\, \dfrac{\delta\nu}{2} \Big\}\,,
\end{equation*}
and assume that the given data satisfy
\begin{equation*}
\big(\dfrac{\nu\delta}{2}\big)^{-2}\Big( c_{g}\Vert \textbf{g}\Vert_{1/2,\partial\Omega}+\Vert\mathbf{f}\Vert_{0,4/3}\Big)\,\leq\, \dfrac{1}{\nu}\,.
\end{equation*}
Then, there exists a unique solution $(\bsi,\mathbf{u})\in\mathbb{X}\times \widehat{\mathbf{Y}}$ for Problem \ref{p:2}, and there holds the following stability estimate
\begin{equation*}
\Vert\bsi\Vert_{\mathbb{X}}+\Vert\mathbf{u}\Vert_{\mathbf{Y}}\,\leq\, \dfrac{2}{\delta}\Big( c_{g}\Vert \mathbf{g}\Vert_{1/2,\partial\Omega}+\Vert\mathbf{f}\Vert_{0,4/3}\Big)\,.
\end{equation*}
\end{Theorem} 
\begin{proof}
See \cite[proof of Theorem 3.8]{Camano21}.
\end{proof}
\section{Weak Galerkin approximation}\label{sec3}
This section aims to introduce the weak Galerkin spaces and the discrete bilinear form essential for introducing a weak Galerkin mixed FEM scheme. We focus on the construction of method in the 2D case for simplicity.
\subsection{Various tools in weak Galerkin method}
A fundamental aspect of the weak Galerkin method lies in employing uniquely defined weak derivatives instead of traditional derivative operators. Our emphasis here centers on the weak divergence operator, a crucial step in introducing our Weak Galerkin technique. To facilitate this discussion, we begin by offering an overview of the mesh structure.\\[2mm]
\textbf{Mesh notation}. Let $\mathcal{K}_{h}=\{K\}$ be a shape regular mesh of domain $\Omega$ that consists of arbitrary polygon elements, where the mesh size $h=\max\{h_{K}\}$, $h_{K}$ is the diameter of element $K$.
The interior and the boundary of any element $K\in\mathcal{K}_{h}$, are represented by $K^{0}$ and $\partial K$, respectively.
 Denote by $\mathcal{E}_{h}$ the set of all edges in $\mathcal{K}_{h}$, and let $\mathcal{E}_{h}^{0}=\mathcal{E}_{h}\backslash \partial\Omega$ be the set of all interior edges. Here is a set of normal directions on $ \mathcal{E}_{h}$:
\begin{equation}\label{eq17}
\mathrm{D}_{h}\,:=\,\bigg\{\mathbf{n}_{e}:~~~~\mathbf{n}_{e}~\text{is unit and outward normal to}~e,~~~~\text{for all}~e\in \mathcal{E}_{h}\bigg\}\,.
\end{equation}
\textbf{Weak divergence operator and weak Galerkin space}. 
It is well known that the weak divergence operator is well-defined for weak matrix-valued functions
 $\bta=\{\bta_{0},\bta_{b}\}$ on the element $K$ such that $\bta_{0}\in \mathbb{L}^{2}(K)$ and $\bta_{b}\mathbf{n}_{e}\in \mathbf{H}^{-1/2}(\partial K)$, where $\mathbf{n}_{e}\in \mathrm{D}_{h}|_{K}$. Components  $\bta_{0}$ and $\bta_{b}$ can be understood as the value of function $\bta$ in $K^{0}$ and on $\partial K$, respectively.
We follow \cite{Wang14}, and introduce for each $K\in\mathcal{K}_{h}$ the local weak tensor space
\begin{equation}\label{eq14}
\mathbb{X}_{\mathtt{w}}(K)\,:=\,\Big\{\bta=\{\bta_{0},\bta_{b}\}:~~~\bta_{0}\in \mathbb{L}^{2}(K)\qan \bta_{b}\mathbf{n}_{e}\in \mathbf{H}^{-1/2}(\partial K)\quad\forall\,\mathbf{n}_{e}\in \mathrm{D}_{h}|_{K}\Big\}\,.
\end{equation}
The global space $\mathbb{X}_{\mathtt{w}}$ is defined by gluing together all local spaces $\mathbb{X}_{\mathtt{w}}(K)$ for any $K\in\mathcal{K}_{h}$. Now, we define
the weak divergence operator for matrix-valued functions as follows.
\begin{Definition}[\cite{Wang13}]
\label{d1a}
For any weak matrix-valued function $\bta\in \mathbb{X}_{\mathtt{w}}(K)$ and element $K\in\mathcal{K}_{h}$, the weak divergence operator, denoted by  $\operatorname{\mathbf{div}}_{\mathtt{w}}$, is defined as the unique vector-valued function $\operatorname{\mathbf{div}}_{\mathtt{w}}(\bta)\in \mathbf{H}^{1}(K)$ satisfying
	\begin{equation}\label{eq15}
	(\bdiv_{\mathtt{w}}(\bta),\,\bze)_{0,K}\,:=\,-(\bta_{0},\,\nabla \bze)_{0,K^{0}}\,+\,\langle\bta_{b}\mathbf{n}_{e}, \bze\rangle_{0,\partial K}\quad\forall\,\bze\in \mathbf{H}^{1}(K)\,.
	\end{equation}
\end{Definition}
Our focus will be on a subspace of $\mathbb{X}_{\mathtt{w}}$ in which $(\bta_{b}|_{e})=(\bta|_{e}\mathbf{n}_{e})\mathbf{n}_{e}$. On the other hand, discrete weak divergence operator can be introduced using a finite-dimensional space $\mathbb{X}_{h}\subset\mathbb{X}_{\mathtt{w}}$, which will be stated in the next. First, for any mesh object $\varpi\in \mathcal{K}_{h}\cup \mathcal{E}_{h}$ and for any $r\in \mathrm{N}$ let us introduce the space $\mathrm{P}_{r}(\varpi)$ to be the space of polynomials defined on $\varpi$ of degree $\leq r$, with the extended notation $\mathrm{P}_{-1}(\varpi)=\{0\}$. Similarly, we let $\mathbf{P}_{r}(\varpi)$ and $\mathbb{P}_{r}(\varpi)$ be the vectorial and tensorial
versions of $\mathrm{P}_{r}(\varpi)$.
Then, given $ k\in \mathrm{N}_0 $, we define for any $ K\in \mathcal{K}_h $ the local discrete weak Galerkin space 
\begin{equation*}
	\begin{split}
\mathbb{X}_{h}(K)\,:=\,\Big\{\bta_{h}&=\{\bta_{0h},\bta_{bh}\}\in \mathbb{X}_{\mathtt{w}}(K):\quad
\bta_{0h}|_{K}\in\mathbb{ P}_{k}(K)\qan\\
&~~~~~~\bta_{bh}|_{e}=\tau_{b}\otimes\mathbf{n}_{e}\,,\,\,\tau_{b}\in \mathbf{P}_{k}(e)\,,\,\,\forall\, e\subset\partial K \,,\,\,\forall\,\mathbf{n}_{e}\in \mathrm{D}_{h}  \Big\}\,.
\end{split}
\end{equation*}
In addition, the global finite dimensional space $\mathbb{X}_{h}$, associated with the partition $\mathcal{K}_{h}$, is defined so that the restriction of every weak function $\bta_{h}$ to the mesh element $K$ belongs to $\mathbb{X}_{h}(K)$.
\begin{Definition}[\cite{Wang13}]
\label{d3} 
For any $\bta_{h}\in \mathbb{X}_{h}(K)$ and element $K\in\mathcal{K}_{h}$, the discrete weak divergence operator, denoted by  $\operatorname{\mathbf{div}}_{\mathtt{w,h}}$, is defined as the unique vector-valued polynomial $\operatorname{\mathbf{div}}_{\mathtt{w,h}}(\bta_{h})\in \mathbf{P}_{k+1}(K)$ satisfying
	\begin{equation}\label{eq18}
(\bdiv_{\mathtt{w,h}}(\bta_{h}),\,\bze_{h})_{0,K}\,:=\,-(\bta_{0h},\,\nabla \bze_{h})_{0,K^{0}}\,+\,\langle\bta_{bh}\mathbf{n}_{e}, \bze_{h}\rangle_{0,\partial K}\quad\forall\,\bze_{h}\in \mathbf{P}_{k+1}(K)\,.
	\end{equation}
\end{Definition}
On the other hand, for approximating the velocity unknowns we simply consider the piecewise polynomial space of degree up to $k+1$:
\[
\mathbf{Y}_h \,:=\,\bigg\{\mathbf{v}_h \in \mathbf{Y}:\quad \mathbf{v}_h |_{K}\in\mathbf{P}_{k+1}(K)\quad \text{for all}~K\in\mathcal{K}_h  \bigg\}\,.
\]

\hspace{-.6cm}\textbf{L$^{2}$-orthogonal projections and approximation properties.} For any $r\in$ N and $ K\in \mathcal{K}_h $, we introduce L$^{2}$-projection operators $\Pcalbb_{0,r}^{K}:\mathbb{L}^{2}(K)\rightarrow\mathbb{P}_{r}(K)$ and $\Pcalbf_{b,r}^{K}:\mathbf{L}^{2}(\partial K)\rightarrow\mathbf{P}_{r}(\partial K)$ which are type of interior and boundary, and are given by 
\begin{equation}
\int_{K}\Pcalbb_{0,r}^{K}(\bta): \widehat{\mathbf{q}}_r=\int_{K}\bta : \widehat{\mathbf{q}}_r\quad\qan\quad\int_{\partial K}\Pcalbf_{b,r}^{K}(\mathbf{v})\cdot\mathbf{q}_r=\int_{\partial K}\mathbf{v} \cdot\mathbf{q}_r \,,
\end{equation}
for all $(\bta, \mathbf{v})\in \mathbb{L}^{2}(K)\times \mathbf{L}^{2}(\partial K)$ and 
$(\widehat{\mathbf{q}}_r,\mathbf{q}_r) \in \mathbb{P}_{r}(K)\times\mathbf{P}_{r}(\partial K)$.\\[2mm]
%
Now, we introduce projection operator $\Pcalbb_{h}^{K}$ into the tensorial weak Galerkin space $\mathbb{X}_{h}(K)$ as:
 \[
 \Pcalbb_{h}^{K}\bta\,:=\,\left\{  \Pcalbb_{0,k}^{K}\bta_0 \,,\,  \Pcalbf_{b,k}^{K}(\tau_b)\otimes\mathbf{n}_{e}\right\},\qquad \text{for all} ~\bta\in \mathbb{X}_{h}(K)\,.
  \]
Also, for each element $K\in\mathcal{K}_h$ and function $\bta\in\mathbb{X}_{h}$, the global projection operator $ \Pcalbb_{h}$ on the space $\mathbb{X}_{h}$ is defined by
\[
\Pcalbb_{h}(\bta)|_{K}\,=\,\Pcalbb_{h}^{{K}}(\bta|_{K})\,.
\]
The approximation properties of $ \Pcalbb_{0} $ and $ \boldsymbol{\mathcal{P}}_{h} $ are stated as follows.
\begin{lemma}\label{l1}
	Let $\mathcal{K}_{h}$ be a finite element partition of $\Omega$ satisfying the shape regularity assumptions $\mathbf{A1}$-$\mathbf{A4}$ stated in \cite{Wang14}. Then, for $k,s,m\in \mathrm{N}_{0}$ such that $m\in \{0,1\}$ there exist constants $C_{1}, C_{2}$, independent on the mesh size $h$, such that
	\begin{align}
	\sum\limits_{K\in \mathcal{K}_{h}}\|\bta-\Pcalbb_{0}^{K}(\bta)\|_{m,0;K}^{2}&\,\leq\, C_{1}h^{2(s-m)}\vert \bta\vert_{s}^{2}\qquad s\leq k \,,\label{eq_Q0}\\
	\sum\limits_{K\in \mathcal{K}_{h}}\|\mathbf{v}-\boldsymbol{\mathcal{P}}_{k+1}^{K}(\mathbf{v})\|_{0;K}^{2}&\,\leq\, C_{2}h^{2s}\vert\mathbf{v}\vert_{s}^{2}\qquad ~~~ s\leq k+1 \,.\label{eq_Q0A}
	\end{align}
\end{lemma}
\subsection{Weak Galerkin scheme}
In order to define our weak Galerkin scheme for Problem \ref{p:2}, we now introduce, when necessary, the discrete versions of the bilinear forms and functionals involving the weak spaces. 
Following the usual procedure in the WG setting, the construction of them is based on the weak derivatives to ensure computability for all weak functions.


Notice that, for each $ \bze,\bta\in \mathbb{X}_h $ and $ \bz,\bw,\bv\in \mathbf{Y}_h $, the quantities
\[
\mathcal{A}(\bze,\bta)\,,\qquad \mathcal{C}(\bz, \bw ; \bta)\,,\qquad \mathcal{G}(\bta)\qan \mathcal{F}(\mathbf{v})\,,
\]
are computable, while the bilinear form $ \mathcal{B}|_{\mathbb{X}_h\times\mathbf{Y}_h} $ is not computable because it involve the divergence operator which cannot be evaluated for weak functions. To overcome this matter, employing Definition \ref{d3} we define the discrete bilinear form
 $ \mathcal{B}_{h}^{K}: \mathbb{X}_{h}(K)\times \mathbf{P}_{k+1}(K)\rightarrow\mathrm{R}$ by
\begin{equation}\label{eq:e3}
\mathcal{B}_{h}^{K}(\bta,\bv)\,:=\,\disp\int_{K}\bdiv_{\mathtt{w,h}}(\bta_h)\cdot \bv_h \qquad \forall\, (\bta,\bv)\in \mathbb{X}_{h}(K)\times\mathbf{P}_{k+1}(K)\,.
\end{equation}
On the other hand, despite the computability of $ \mathcal{A}|_{\mathbb{X}_h\times\mathbb{X}_h} $, this form needs an additional stabilizer term to achieve the well-posedness of our weak Galerkin scheme. More precisely,  we define the corresponding discrete bilinear form as follows:
\begin{equation}\label{eq:e2}
\mathcal{A}_{h}^{K}(\bze,\bta)\,:=\,\dfrac{1}{\nu}\disp\int_{K}\bsi_{0}^{\mathtt{d}}:\bta_{0}^{\mathtt{d}}\,+\,\rho\,\mathcal{S}^{K}(\bze,\bta)\qquad\forall\,\bze,\bta\in \mathbb{X}_{h}(K)\,,
\end{equation}
where $\rho$ is the piecewise constant on $\mathcal{K}_h$
and the stabilization form $\mathcal{S}^{K}(\cdot,\cdot): \mathbb{X}_{h}(K)\times \mathbb{X}_{h}(K)\rightarrow \mathrm{R}$ is given by
\begin{equation}\label{def:Stab}
\mathcal{S}^{K}(\bze,\bta)\,:=\, h_{K}\left\langle \bze_{0}\mathbf{n}-\bze_{b}\mathbf{n},\, \bta_{0}\mathbf{n}-\bta_{b}\mathbf{n}\right\rangle_{0,\partial K}\qquad\forall\,\bze,\bta\in \mathbb{X}_{h}(K)\,.
\end{equation}\\
In addition, the global bilinear forms $\mathcal{A}_{h}$ and $\mathcal{B}_{h}$  can be derived by adding the local contributions, that is,
\[
\mathcal{A}_{h}(\cdot ,\cdot)\,:=\,\sum\limits_{K\in \mathcal{K}_{h}}\mathcal{A}_{h}^{K}(\cdot ,\cdot)\,\qan \,
\mathcal{B}_{h}(\cdot ,\cdot)\,:=\,\sum\limits_{K\in \mathcal{K}_{h}}\mathcal{B}_{h}^{K}(\cdot ,\cdot)\,.
\]
Finally, let us introduce the subspace of $\mathbb{X}_{h}$ as
\[
\mathbb{X}_{0,h}\,:=\,\bigg\{ \bta_{h}=\{\bta_{0h},\bta_{bh}\}\in \mathbb{X}_{h}:\quad\disp\int_{\Omega}\tr(\bta_{0h})\,=\,0\bigg\}\,.
\]
Referring to the above space, the discrete bilinear form \eqref{eq:e2} and \eqref{eq:e3}, 
the discrete weak Galerkin problem reads as follow.
\begin{prob}[WG problem]
\label{p:3}
Find $\bsi_h \in\mathbb{X}_{0,h}$ and $\mathbf{u}_h \in \mathbf{Y}_h$ such that
\begin{equation*}
\left\{\begin{array}{ll}
\mathcal{A}_h(\bsi_h, \bta_h)+\mathcal{C}(\bu_h,\bu_h;\bta_h)+\mathcal{B}_h(\bta_h,\bu_h)\,=\,\mathcal{G}(\bta_h)\qquad\forall\,\bta_h\in \mathbb{X}_{0,h}\,,  \\[3mm]
\hspace{5.3cm}\mathcal{B}_h(\bsi_h,\bv_h)\,=\,\mathcal{F}(\bv_h)\qquad\forall\,\bv_h\in \mathbf{Y}_h\,. 
\end{array}\right.
\end{equation*}
\end{prob}
\section{Solvability analysis}\label{sec4}
This section goals to examine the solvability of Problem \ref{p:3}. Initially, we delve into the discussion of the stability properties inherent in the discrete forms of Problem \ref{p:3}. Subsequently, we employ the Banach fixed-point and classical Banach--Nečas--Babuška theorems to establish the well-posedness of the discrete scheme, assuming an appropriate smallness condition on the data.
\subsection{Stability properties}\label{sub.sec4.1}
We begin by equipping the approximate pair spaces $\mathbb{X}_{h}$ and $\mathbf{Y}_h+\mathbf{H}^{1}(\Omega)$ with the following discrete norms, respectively (see e.g., \cite{Wang14})
\[
\Vert \bta_{h}\Vert_{\mathbb{H}, h}^{2}\,:=\,\sum_{K\in\mathcal{K}_{h}}\big[\Vert \bta_{0h}\Vert_{0,K}^{2}+\mathcal{S}^{K}(\bta_{h},\bta_{h})\big]\qquad \text{for all} ~\bta_{h}\in \mathbb{X}_{h}\,,
\] 
and
\[
\Vert \mathbf{v}_h\Vert_{1,h}^{2}\,:=\,\sum_{K\in\mathcal{K}_{h}}\vert \mathbf{v}_h\vert_{1,K}^{2}+\sum_{e\in\mathcal{E}_{h}}h_{e}^{-1}\Vert \Pcalbf_{b,k}\llbracket \mathbf{v}_h\rrbracket\Vert_{0,e}^{2}\qquad\quad \text{for all} ~\mathbf{v}_{h}\in \mathbf{Y}_{h}+\mathbf{H}^{1}(\Omega)\,,
\]
Next, we provide the following result which is a counterpart of \cite[Lemma 3.1]{Camano21} for the weak divergence operator. 
\begin{lemma}\label{lb}
 There exists $\hat{c}_{\Omega}>0$ depends on $\Omega$ but independent of mesh size, such that
\begin{equation}
\hat{c}_{\Omega}\Vert\bta_{0h}\Vert_{0}^{2}\,\leq\, \Vert\bta_{0h}^{\mathtt{d}}\Vert_{0}^{2}+\Vert\operatorname{\mathbf{div}}_{\mathtt{w,h}}(\bta_h)\Vert_{0,4/3}^{2}\quad\quad\forall\, \bta_h\in \mathbb{X}_{h}\,.
\end{equation}
\end{lemma}
\begin{proof}
See \cite[Lemma 3.1]{Gharibi24}.
\end{proof}
Then, some properties of $ \mathcal{A}_h $ (cf. \eqref{eq:e2}) are established as follows.
\begin{lemma}\label{l_ah}
The discrete bilinear form $\mathcal{A}_h$ defined in \eqref{eq:e2}, satisfy the following properties:
\begin{itemize}
\item[•] \underline{consistency}: for any $\boldsymbol{\xi}\in\mathbb{H}^{1}(\Omega)$, we have that
\[
\mathcal{A}_h(\bze_{h}, \boldsymbol{\xi})\,=\,\mathcal{A}(\bze_{h}, \boldsymbol{\xi})\qquad \forall\,\bze_{h}\in \mathbb{X}_{h} \,.
\]
\item[•] \underline{stability and boundedness}: there exists positive constant $c_{\mathcal{A}}$, independent of $K$ and $h$, such that:
\begin{equation}\label{bun_Ah}
\big|\mathcal{A}_{h}(\bze_{h},\bta_{h})\big|\,\leq\, c_{\mathcal{A}} \,\Vert\bze_{h}\Vert_{\mathbb{H},h}\,\Vert\bta_{h}\Vert_{\mathbb{H},h}\qquad\forall\,\bze_{h},\bta_{h}\in \mathbb{X}_{h} \,,
\end{equation}
and let $\widetilde{\mathbb{X}}_{h}$ be the discrete kernel of the bilinear form $\mathcal{B}_h$. Then, there exists constant $\alpha>0$, independent of $K$, such that
\begin{equation}\label{co_Ah}
\mathcal{A}_{h}(\bta_h,\bta_h)\,\geq\, \alpha \, \Vert\bta_h\Vert_{\mathbb{H},h}^{2}\qquad\forall\,\bta_h\in \widetilde{\mathbb{X}}_{h}\,.
\end{equation}
\end{itemize}
\end{lemma}
\begin{proof}
By considering $\boldsymbol{\xi}\in\mathbb{H}^{1}(\Omega)$ and $\bze_{h}\in \mathbb{X}_{h}$ and employing the definition of the discrete form $\mathcal{A}_{h}$ (cf. \eqref{eq:e2}), along with the observation that the stabilization term (cf. \eqref{def:Stab}) vanishes when one of the components is sufficiently regular, we deduce the "consistency" property. Next, to verify the boundedness of $\mathcal{A}_h$, we use the Cauchy--Schwarz inequality and virtue of
\begin{equation}\label{taud}
\Vert\bta_{0h}^{\mathtt{d}}\Vert_{0,K}^{2}\,=\,\Vert\bta_{0h}\Vert_{0,K}^{2}-\dfrac{1}{2}\Vert\operatorname{tr}(\bta_{0h})\Vert_{0,K}^{2}\,\leq\, \Vert\bta_{0h}\Vert_{0,K}^{2}\,.
\end{equation}
This gives
\begin{align*}
\big|\mathcal{A}_{h}(\bze_{h},\bta_{h})\big|&\,=\,\bigg|\sum_{K\in\mathcal{K}_{h}}\Big(\dfrac{1}{\nu}\disp\int_{K}\bze_{0h}^{\mathtt{d}}:\bta_{0h}^{\mathtt{d}}+\rho\,\mathcal{S}^{K}(\bze_{h},\bta_{h})\Big)\bigg|\nonumber\\
&\,\leq\, \sum_{K\in\mathcal{K}_{h}}\Big(\dfrac{1}{\nu}\Vert\bze_{0h}\Vert_{0,K}\,\Vert\bta_{0h}\Vert_{0,K}+\rho\left(\mathcal{S}^{K}(\bze_{h},\bze_{h})\right)^{1/2}\left(\mathcal{S}^{K}(\bta_{h},\bta_{h})\right)^{1/2}\Big)\\
&\,\leq\, \max\left\{\dfrac{1}{\nu},\rho\right\}\Vert\bze_{h}\Vert_{\mathbb{H},h}\,\Vert\bta_{h}\Vert_{\mathbb{H},h}\,.
\end{align*}

We now aim to establish the ellipticity of $\mathcal{A}_h$ on the discrete kernel of $\mathcal{B}_{h}|_{\mathbb{X}_h\times\mathbf{Y}_h}$, that is,
\[
\widetilde{\mathbb{X}}_{h} \,:=\,\bigg\{ \bta_h \in\mathbb{X}_{h} :\quad  \int_{\Omega}\operatorname{\mathbf{div}}_{\mathtt{w,h}}(\bta_h)\cdot \mathbf{v}_h \,=\,0 \quad \forall\, \mathbf{v}_h\in\mathbf{Y}_h \bigg\}\,,
\]
which, together with fact $\operatorname{\mathbf{div}}_{\mathtt{w,h}}(\bta_h)\in\mathbf{Y}_h$ (cf. Definition \ref{d3}) implies that
\begin{equation}\label{KerD}
\widetilde{\mathbb{X}}_{h} \,:=\,\Big\{ \bta_h \in\mathbb{X}_{h} :\quad  \bdiv_{\mathtt{w,h}}(\bta_h) \,=\,0 \quad\text{in}\,\,\Omega\,\Big\}\,.
\end{equation}
Therefore, employing Lemma \ref{lb} and the fact that $\bta_h$ is weak divergence-free, we obtain
\begin{align*}
\mathcal{A}_{h}(\bta_{h},\bta_{h})&\,=\,\sum_{K\in\mathcal{K}_{h}}\Big(\dfrac{1}{\nu}\Vert\bta_{0h}^{\mathtt{d}}\Vert_{0,K}^{2}+\rho\,\mathcal{S}^{K}(\bta_{h},\bta_{h})\Big)\\
&\,\geq\, \sum_{K\in\mathcal{K}_{h}}\Big(\dfrac{\hat{c}_{\Omega}}{\nu}\Vert\bta_{0h}\Vert_{0,K}^{2}+\rho\,\mathcal{S}^{K}(\bta_{h},\bta_{h})\Big)\\
&\,\geq\, \min\left\{\dfrac{\hat{c}_{\Omega}}{\nu},\rho\right\}\Vert\bta_{h}\Vert_{\mathbb{H},h}^{2}\,,
\end{align*}
which together with setting $c_{\mathcal{A}} \,=\,\max\left\{\dfrac{1}{\nu},\rho\right\}$ and $\alpha \,=\,\min\left\{\dfrac{1}{\nu},\rho\right\}$ completes the proof.
\end{proof}
To establish the discrete inf-sup condition for the bilinear form $\mathcal{B}_h$, we require a preliminary result, as stated in the following lemma.
\begin{lemma}\label{l_BBh}
For any $\bta_{h}\in \mathbb{X}_{h}$ and $\mathbf{v}_{h}\in \mathbf{Y}_{h}$, we have
\[
\mathcal{B}_{h}(\bta_{h},\mathbf{v}_{h})\,=\,\sum_{e\in\mathcal{E}_h}\langle \bta_{bh}\mathbf{n}_{e},  \llbracket \mathbf{v}_h\rrbracket\rangle_{e}-\sum_{K\in\mathcal{K}_h}(\bta_{0h}, \boldsymbol{\nabla}\mathbf{v}_h)_{0,K}\,.
\]
\end{lemma}
\begin{proof}
It straightforwardly follows from the definition of $\mathcal{B}_h$, as given in \eqref{eq:e3}, and the application of the discrete divergence operator (cf. Definition \ref{d3}).
\end{proof}
We are now in a position to establish the discrete inf-sup condition for the bilinear from $\mathcal{B}_h$.
\begin{lemma}\label{l_iB}
There exists a positive constant $\widehat{\beta}$ independent of $h$ such that
\begin{equation}\label{infD}
\sup_{\mathbf{0}\neq \bta_{h}\in \mathbb{X}_{h}}\dfrac{\mathcal{B}_{h}(\bta_{h},\mathbf{v}_{h})}{\Vert\bta_{h}\Vert_{\mathbb{H},h}}\,\geq\, \widehat{\beta}\,\Vert\mathbf{v}_{h}\Vert_{1,h}\qquad \forall\,\mathbf{v}_{h}\in \mathbf{Y}_{h}\,.
\end{equation}
\end{lemma}
\begin{proof}
In what follows, we proceed similarly to the proof of \cite[Lemma 3.3]{Wang14}. In fact, given $\mathbf{v}_h\in\mathbf{Y}_h$, we set
\begin{equation*}
\left\{\begin{array}{l}
\bta_{0h}\,=\,-\boldsymbol{\nabla}\mathbf{v}_h\qquad\qquad \qquad~\text{in}~K \,,\\[2mm]
\bta_{bh}\,=\,h_{e}^{-1} \Pcalbf_{b,k}\llbracket \mathbf{v}_h\rrbracket \otimes\mathbf{n}_{e}\qquad \text{on}~e \,,
\end{array}\right.
\end{equation*}
which satisfies
\begin{equation}\label{eq1:infsup}
\mathcal{B}_{h}(\bta_{h},\mathbf{v}_{h})\,=\,\sum_{e\in\mathcal{E}_h}\langle \Pcalbf_{b,k}\llbracket \mathbf{v}_h\rrbracket,  \Pcalbf_{b,k}\llbracket \mathbf{v}_h\rrbracket\rangle_{e}+\sum_{K\in\mathcal{K}_h}(\boldsymbol{\nabla}\mathbf{v}_h, \boldsymbol{\nabla}\mathbf{v}_h)_{0,K}\,=\,\Vert \mathbf{v}_h\Vert_{1,h}^{2}\,.
\end{equation}
On the other hand, we use the definition of the discrete norm $\Vert\cdot\Vert_{\mathbb{H},h}$ to the above chosen $\bta_{h}\,:=\,\{\bta_{0h},\bta_{bh}\}$ and the trace inequlaity, to obtain
\begin{align*}
\Vert\bta_{h}\Vert_{\mathbb{H},h}^2&\,=\,\sum_{K\in\mathcal{K}_{h}}\Big(\Vert \boldsymbol{\nabla}\mathbf{v}_h\Vert_{0,K}^{2}+\mathcal{S}^{K}(\bta_{h},\bta_{h})\Big)\nonumber\\
&\,=\,\sum_{K\in\mathcal{K}_{h}}\Big(\Vert \boldsymbol{\nabla}\mathbf{v}_h\Vert_{0,K}^{2}+h_{K}\Vert (\boldsymbol{\nabla}\mathbf{v}_h)\mathbf{n}+h_{e}^{-1} \Pcalbf_{b,k}\llbracket \mathbf{v}_h\rrbracket\Vert_{0,\partial K}^{2}\Big)\nonumber\\
&\,\leq\, (1+\mathcal{C}_{\mathtt{tr}})\sum_{K\in\mathcal{K}_{h}}\Vert \boldsymbol{\nabla}\mathbf{v}_h\Vert_{0,K}^{2}+\sum_{e\in\mathcal{E}_h}h_{e}^{-1} \Vert\Pcalbf_{b,k}\llbracket \mathbf{v}_h\rrbracket\Vert_{0,e}^{2}\,\leq\, \max\{1,(1+\mathcal{C}_{\mathtt{tr}})  \}\Vert \mathbf{v}_h\Vert_{1,h}^{2}\,.
\end{align*}
Combining the above result with \eqref{eq1:infsup} concludes the desired inequality \eqref{infD} with setting $\widehat{\beta}:=\frac{1}{\sqrt{\max\{1,(1+\mathcal{C}_{\mathtt{tr}})  \}}}$.
\end{proof}
The boundedness properties of the discrete bilinear form $ \mathcal{B}_h $ (cf. \eqref{eq:e3}) is established as follows.
\begin{lemma}\label{l_bundBh}
The discrete bilinear form $\mathcal{B}_h(\cdot,\cdot)$ is bounded in $\mathbb{X}_{h}\times\mathbf{Y}_h$. In other words, there exists a positive constant $c_{\mathcal{B}}$ such that
\begin{equation}
\big| \mathcal{B}_h(\bta_h,\mathbf{v}_h)\big|\,\leq\, c_{\mathcal{B}} \,\Vert \bta_h\Vert_{\mathbb{H},h}\,\Vert \mathbf{v}_h\Vert_{1,h}\qquad\forall\, (\bta_h,\mathbf{v}_h) \in \mathbb{X}_{h} \times\mathbf{Y}_h \,.
\end{equation}
\end{lemma}
\begin{proof}
An application of Lemma \ref{l_BBh} and Cauchy--Schwarz inequality, yields
\begin{equation}\label{eq2:boundBh}
	\begin{array}{c}
\big| \mathcal{B}_h(\bta_h,\mathbf{v}_h)\big|\,\leq\, \disp\sum_{e\in\mathcal{E}_h}\Vert \bta_{bh}\mathbf{n}_{e}\Vert_{0,e}\,\Vert  \llbracket \mathbf{v}_h\rrbracket\Vert_{0,e}+\disp\sum_{K\in\mathcal{K}_h}\Vert\bta_{0h}\Vert_{0,K}\, \Vert\boldsymbol{\nabla}\mathbf{v}_h\Vert_{0,K}\\[2ex]
\,\leq\, \left(\disp\sum_{e\in\mathcal{E}_h}h_e \Vert \bta_{bh}\mathbf{n}_{e}\Vert_{0,e}^2 \right)^{1/2}
\left(\disp\sum_{e\in\mathcal{E}_h}h_e^{-1}\Vert  \llbracket \mathbf{v}_h\rrbracket\Vert_{0,e}^{2} \right)^{1/2}+\Vert\bta_{0h}\Vert_{0,\Omega}\left(\disp\sum_{K\in\mathcal{K}_h}\Vert\boldsymbol{\nabla}\mathbf{v}_h\Vert_{0,K}^{2}\right)^{1/2}\\[2ex]
\,\leq\, \bigg\{ \left(\disp\sum_{e\in\mathcal{E}_h}h_e \Vert \bta_{bh}\mathbf{n}_{e}\Vert_{0,e}^2 \right)^{1/2}+ \Vert\bta_{0h}\Vert_{0,\Omega}\bigg\}\,\Vert \mathbf{v}_h\Vert_{1,h}\,.
	\end{array}
\end{equation}
Our aim is now to show that
\[
\left(\sum_{e\in\mathcal{E}_h}h_e \Vert \bta_{bh}\mathbf{n}_{e}\Vert_{0,e}^2 \right)^{1/2}+ \Vert\bta_{0h}\Vert_{0,\Omega}\,\lesssim \,\Vert \bta_h\Vert_{\mathbb{H},h}\qquad\forall\,\bta_{h}\in\mathbb{X}_h \,.
\]
To attain the aforementioned inequality, we employ the trace inequality, resulting 
\begin{align*}
\sum_{e\in\mathcal{E}_h}h_e \Vert \bta_{bh}\mathbf{n}_{e}\Vert_{0,e}^2&\,\leq\, 2\sum_{e\in\mathcal{E}_h}\Big(h_e \Vert (\bta_{bh}-\bta_{0h})\mathbf{n}_{e}\Vert_{0,e}^2+h_e \Vert \bta_{0h}\mathbf{n}_{e}\Vert_{0,e}^2\Big)\nonumber\\[1mm]
&\,\leq\, 2\sum_{e\in\mathcal{E}_h}h_e \Vert (\bta_{bh}-\bta_{0h})\mathbf{n}_{e}\Vert_{0,e}^2+2\,\mathcal{C}_{\mathtt{tr}}^2\sum_{K\in\mathcal{K}_h}\Vert \bta_{0h}\Vert_{0,K}^2 \,.
\end{align*}
Substituting the above result back into \eqref{eq2:boundBh} finishes the proof of lemma.
\end{proof}
Now, employing the aforementioned stability properties of $\mathcal{A}_{h}$ (cf. \ref{l_ah}, eqs. \eqref{bun_Ah} and \eqref{co_Ah}), the discrete inf-sup condition and the boundedness of $\mathcal{B}_h$ (cf. Lemmas \ref{l_iB} and \ref{l_bundBh}), and applying \cite[Proposition 2.36]{eg-AMS-2004} it is not difficult to see that the bilinear form $\Lambda:(\mathbb{X}_h\times\mathbf{Y}_h)\times(\mathbb{X}_h\times\mathbf{Y}_h)\rightarrow\mathrm{R}$  defined by
\begin{equation}\label{eq:defGamF}
\Lambda_h[(\bze_{h},\textbf{w}_{h}), (\bta_{h},\bv_{h})]\,:=\,\mathcal{A}_{h}(\bze_{h},\bta_{h})+\mathcal{B}_h(\bta_{h},\textbf{w}_{h})-\mathcal{B}_h(\bze_{h},\bv_{h})\,,
\end{equation}
satisfies:
\begin{equation}\label{eq:infsupF}
\sup_{\textbf{0}\neq (\bta_{h},\mathbf{v}_{h})\in \mathbb{X}_{h}\times \mathbf{Y}_{h}}\dfrac{\Lambda_h[(\bze_{h},\mathbf{w}_{h}), (\bta_{h},\mathbf{v}_{h})]}{\Vert(\bta_{h},\mathbf{v}_{h})\Vert_{\mathbb{X}\times \mathbf{Y}}}\,\geq\, \alpha_{\Lambda,\mathtt{d}}\,\Vert(\bze_{h},\mathbf{w}_{h})\Vert_{\mathbb{X}\times \mathbf{Y}}\,,
\end{equation}
 where $\alpha_{\Lambda,\mathtt{d}}$ is the positive constant dependent on $c_{\mathcal{A}}, c_{\mathcal{B}}, \alpha,  \widehat{\beta}$.
 
Furthermore, the boundedness of the trilinear form $\mathcal{C}(\cdot;\cdot,\cdot)$ on $\mathbf{Y}_{h}\times \mathbf{Y}_{h}\times\mathbb{X}_{h}$ can be readily inferred by utilizing the H\"older inequality, the Sobolev embedding $\mathbf{H}^1\subset\mathbf{L}^4$ and inequality \eqref{taud}, as follows:
\begin{equation}\label{boundC}
\begin{array}{c}
\big|\mathcal{C}(\textbf{w}_{h};\bv_{h},\bta_{h})\big|\,=\,\big|\dfrac{1}{\nu}\left((\textbf{w}_{h}\otimes\bv_{h})^{\mathtt{d}},\bta_{0h}\right)_{0,\Omega}\big|\\[2ex]
\,=\,\big|\dfrac{1}{\nu}\left((\textbf{w}_{h}\otimes\bv_{h}),\bta_{0h}^{\mathtt{d}}\right)_{0,\Omega}\big|\,\leq\, \dfrac{1}{\nu}\Vert\textbf{w}_{h}\Vert_{\mathbf{Y}}\Vert\bv_{h}\Vert_{\mathbf{Y}}\Vert\bta_{0h}^{\mathtt{d}}\Vert_{0,\Omega}\\[2ex]
\,\leq\, \dfrac{c_{\mathtt{em}}^{2}}{\nu}\Vert\textbf{w}_{h}\Vert_{1,h}\Vert\bv_{h}\Vert_{1,h}\Vert\bta_{h}\Vert_{\mathbb{H},h}\,.
\end{array}
\end{equation}
\subsection{The fixed-point strategy}
We begin by introducing the associated fixed-point operator for any $\textbf{z}_{h}\in \mathbf{Y}_{h}$ as $\mathbf{S}_{\mathtt{d}}(\textbf{z}_{h})\,=\,\bu_{\star,h}$, where $(\bsi_{\star,h},\bu_{\star,h})\in \mathbb{X}_{0,h}\times\mathbf{Y}_h$ is the solution of the linearized version of Problem \ref{p:3}, that is,
\begin{equation}\label{eq:linProb}
\left\{\begin{array}{rcll}
\mathcal{A}_{h}(\bsi_{\star,h}, \bta_{h})+\mathcal{C}(\textbf{z}_{h},\bu_{\star,h};\bta_{h})+\mathcal{B}_{h}(\bta_{h},\bu_{\star,h})\,&=&\,\mathcal{G}(\bta_{h}) & \quad\forall\,\bta_{h}\in \mathbb{X}_{0,h}\,,\\[3mm]
\mathcal{B}_{h}(\bsi_{\star,h},\bv_{h})\,&=&\,\mathcal{F}(\bv_{h}) &\quad \forall\, \bv_{h}\in \mathbf{Y}_{h}\,.
\end{array}\right.
\end{equation}

Equivalently, given $ \bz_h\in \mathbf{Y}_h $, introducing the bilinear form $\Lambda_{h,\bz_h}:(\mathbb{X}_h\times\mathbf{Y}_h)\times(\mathbb{X}_h\times\mathbf{Y}_h)\rightarrow\mathrm{R}$ given by 
\begin{equation}\label{eq:defhatGamF}
\Lambda_{h,\bz_{h}}[(\bze_{h},\textbf{w}_{h}), (\bta_{h},\bv_{h})]\,:=\,\Lambda_h[(\bze_{h},\textbf{w}_{h}), (\bta_{h},\bv_{h})]\,+\,\mathcal{C}(\textbf{z}_{h};\textbf{w}_{h},\bv_{h}) \,,
\end{equation}
and the linear functional $ \mathrm{F}: \mathbb{X}_{h}\times\mathbf{Y}_h\rightarrow\mathrm{R} $ as 
\begin{equation}\label{key}
\mathrm{F}(\bta_{h},\bv_{h})\,:=\,	\mathcal{G}(\bta_{h})+\mathcal{F}(\bv_h) \,.
\end{equation}
for all $ (\bze_h ,\bw_h), (\bta_{h},\bv_{h})\in \mathbb{X}_{h}\times\mathbf{Y}_h $, the linearized problem \eqref{eq:linProb} can be rewritten as
\begin{equation}\label{eqF:sub1}
\Lambda_{\textbf{z}_{h}}[(\bsi_{\star,h},\textbf{u}_{\star,h}), (\bta_{h},\bv_{h})]\,=\,\mathrm{F}(\bta_{h},\bv_{h})\qquad\forall\,  (\bta_{h},\bv_{h})\in \mathbb{X}_{0,h}\times\mathbf{Y}_h \,.
\end{equation}

It can be observed that solving Problem \ref{p:3} is equivalent to seeking a fixed point of $\textbf{S}_{\mathtt{d}}$, that is: Find $\textbf{u}_h\in\mathbf{Y}_h$ such that
$$ \textbf{S}_{\mathtt{d}}(\textbf{u}_h)\,=\,\textbf{u}_h \,. $$

Next, we utilize the classical Banach--Nečas--Babuška theorem (see, for instance \cite[Theorem 2.6]{eg-AMS-2004}) to demonstrate that, for any arbitrary $\textbf{z}_h\in\mathbf{Y}_h$, the problem \eqref{eq:linProb} (or equivalently \eqref{eqF:sub1}) is well-posed, implying the well-definedness of $ \mathbf{S}_{\mathtt{d}} $.

The following lemma demonstrates the global discrete inf-sup condition for form $\Lambda_{\bz_{h}}$ on pair-space $\mathbb{X}_{h}\times \mathbf{Y}_{h}$.
\begin{lemma}
\label{l_iAC}
For any $\mathbf{z}_{h}\in\mathbf{Y}_{h} $ such that 
\begin{equation}\label{eq:as1}
	\Vert\mathbf{z}_{h}\Vert_{1,h}\,\leq\, \dfrac{\nu\,\alpha_{\Lambda,\mathtt{d}}}{2c_{\mathtt{em}}^{2}}\,,
\end{equation}
 there holds
\begin{equation}\label{QQ2}
\sup_{\mathbf{0}\neq (\bta_{h},\mathbf{v}_{h})\in \mathbb{X}_{h}\times \mathbf{Y}_{h}}\dfrac{\Lambda_{h,\bz_{h}}[(\bze_{h},\mathbf{w}_{h}), (\bta_{h},\mathbf{v}_{h})]}{\Vert(\bta_{h},\mathbf{v}_{h})\Vert_{h}}\,\geq\, \dfrac{\alpha_{\Lambda,\mathtt{d}}}{2}\,\Vert(\bze_{h},\mathbf{w}_{h})\Vert_{h}\qquad \forall\,(\bze_{h},\mathbf{w}_{h})\in\mathbb{X}_{h}\times \mathbf{Y}_{h}\,.
\end{equation}
where $\alpha_{\Lambda,\mathtt{d}}$ is the positive constant in \eqref{eq:infsupF}.
\end{lemma}
\begin{proof}
Bearing in mind the definition of $\Lambda_{h,\bz_{h}}$ (cf. \eqref{eq:defhatGamF}) for each $\bz_h \in \mathbf Y_h$, 
and combining \eqref{eq:infsupF} with the 
effect of the extra term given by $\mathcal{C}(\bz_h;\cdot,\cdot)$, which means invoking the upper bound provided by  \eqref{boundC}, 
we arrive at
\begin{equation*}
\sup_{\mathbf{0}\neq (\bta_{h},\mathbf{v}_{h})\in \mathbb{X}_{h}\times \mathbf{Y}_{h}}\dfrac{\Lambda_{h,\bz_{h}}[(\bze_{h},\mathbf{w}_{h}), (\bta_{h},\mathbf{v}_{h})]}{\Vert(\bta_{h},\mathbf{v}_{h})\Vert_{h}}\,\geq\, \Big\{\alpha_{\Lambda,\mathtt{d}} - \dfrac{c_{\mathtt{em}}^{2}}{\nu} \,\|\bz_h\|_{\mathbf Y}\Big\}
\Vert(\boldsymbol{\zeta}_h, \bw_h)\Vert_{h}
\quad \forall ~(\bze_h,\bw_h)\in \mathcal{V}_{h},
\end{equation*}
from which, under the assumption \eqref{eq:as1} we arrive at \eqref{QQ2}, which ends the proof.
\end{proof}
Now we ready to show that mapp $\mathbf{S}_{\mathtt{d}}$ is well-defined or equivalently problem \eqref{eq:linProb} is uniquely solvable.
\begin{lemma}
\label{wp_S}
Let the assumption of Lemma \ref{l_iAC} be satisfied. Then, there exists a unique $(\bsi_{\star,h}, \bu_{\star,h})\in\mathbb{X}_h\times\mathbf{Y}_h$ solution to problem \eqref{eq:linProb}. In addition, there holds
 \begin{equation}\label{Stab}
 \Vert \mathbf{S}_{\mathtt{d}}(\mathbf{z}_{h})\Vert_{1,h}\,=\,\Vert\bu_{\star,h}\Vert_{1,h}\,\leq\, \dfrac{2}{\alpha_{\Lambda,\mathtt{d}}} \left(\Vert\mathbf{f}\Vert_{0,4/3}+c_{g}\Vert\mathbf{g}\Vert_{1/2,\partial\Omega} \right)\,.
 \end{equation}
\end{lemma}
\begin{proof}
A straightforward application of the classical Babuška--Brezzi theory and Lemma \ref{l_iAC} implies that problem \eqref{eq:linProb} is well-posed. For the second part of proof, by combining \eqref{eqF:sub1} and Lemma  \ref{l_iAC} with considering $(\bze_{h},\mathbf{w}_{h})\,:=\,(\bsi_{\star,h},\bu_{\star,h})$,
 we readily obtain 
 \begin{align*}
\dfrac{\alpha_{\Lambda,\mathtt{d}}}{2}\,\Vert(\bsi_{\star,h},\bu_{\star,h})\Vert_{h}&\,\leq\,  \sup_{\mathbf{0}\neq (\bta_{h},\mathbf{v}_{h})\in \mathbb{X}_{h}\times \mathbf{Y}_{h}}\dfrac{\Lambda_{h,\bz_{h}}[(\bsi_{\star,h},\bu_{\star,h}), (\bta_{h},\mathbf{v}_{h})]}{\Vert(\bta_{h},\mathbf{v}_{h})\Vert_{h}}\\[1mm]
&\,=\,\sup_{\mathbf{0}\neq (\bta_{h},\mathbf{v}_{h})\in \mathbb{X}_{h}\times \mathbf{Y}_{h}}\dfrac{\mathrm{F}(\bta_{h},\bv_h)}{\Vert(\bta_{h},\mathbf{v}_{h})\Vert_{h}}\,\leq\, \left( c_{g}\Vert \textbf{g}\Vert_{1/2,\partial\Omega}+\Vert\mathbf{f}\Vert_{0,4/3}\right)\,,
 \end{align*}
 where the boundness of $\mathcal{G}$ and $\mathcal{F}$ was used in the last step, and this finishes the proof.
\end{proof}
We continue the analysis by establishing sufficient conditions under which $\mathbf{S}_{\mathtt{d}}$ maps a closed ball of
$\mathbf{Y}_{h}$ into itself.
 Let us define the set
\begin{equation*}
\widehat{\mathbf{Y}}_{h}\,:=\,\bigg\{ \textbf{z}_{h}\in \mathbf{Y}_{h}:\quad \Vert\textbf{z}_{h}\Vert_{1,h}\,\leq\, \dfrac{\nu\,\alpha_{\Lambda,\mathtt{d}}}{2c_{\mathtt{em}}^{2}}\bigg\}\,,
\end{equation*}
and state the following result.
\begin{lemma}\label{l_P2}
Assume that the data are sufficiently small so that
\begin{equation*}
\left( c_{g}\Vert \mathbf{g}\Vert_{1/2,\partial\Omega}+\Vert\mathbf{f}\Vert_{0,4/3}\right)\,\leq\, \dfrac{\nu\,\alpha_{\Lambda,\mathtt{d}}^{2}}{4c_{\mathtt{em}}^{2}}\,,
\end{equation*}
Then $\mathbf{S}_{\mathtt{d}}(\widehat{\mathbf{Y}}_{h})\subseteq \widehat{\mathbf{Y}}_{h}$.
\end{lemma}
\begin{proof}
It deduces straightly from a priori estimate stated by \eqref{Stab}.
\end{proof}
We now address the continuity properties of $ \mathbf{S}_{\mathtt{d}}$.
\begin{lemma}
\label{l_P3} 
For any $ \mathbf{z}_h , \mathbf{y}_h \in \mathbf{Y}_h $, there holds
\begin{equation}
\Vert\mathbf{S}_{\mathtt{d}}(\mathbf{z}_{h})-\mathbf{S}_{\mathtt{d}}(\mathbf{y}_{h})\Vert_{1,h}\,\leq\, \dfrac{4c_{\mathtt{em}}^{2}}{\nu\,\alpha_{\Lambda,\mathtt{d}}^2}\,\left(\Vert\mathbf{f}\Vert_{0,3/4}+c_{g}\Vert\mathbf{g}\Vert_{1/2,\partial\Omega} \right)\Vert\mathbf{z}_{h}-\mathbf{y}_{h}\Vert_{1,h}\,.
\end{equation}
\end{lemma}
\begin{proof}
Given $\bz_{h}, \by_{h}\in \widehat{\mathbf{Y}}_{h}$ we let $\mathbf{S}_{\mathtt{d}}(\textbf{z}_{h})=\bu_{\star,h}$ and $\mathbf{S}_{\mathtt{d}}(\textbf{y}_{h})=\bu_{\circ,h}$, where $(\bsi_{\star,h}, \bu_{\star,h})$ and $(\bsi_{\circ,h}, \bu_{\circ,h})$ are the corresponding solutions of equation \eqref{eqF:sub1}. It follows from \eqref{eqF:sub1} that
\begin{equation*}
\Lambda_{h,\textbf{z}_{h}}[(\bsi_{\star,h}, \bu_{\star,h}), (\bta_{h},\bv_{h})]\,=\,\Lambda_{h,\textbf{y}_{h}}[(\bsi_{\circ,h}, \by_{\circ,h}), (\bta_{h},\bv_{h})]\,,\quad\quad\forall\, (\bta_{h},\bv_{h})\in\mathbb{X}_{h}\times \mathbf{Y}_{h}\,,
\end{equation*}

from which, according to the definitions of $\Lambda_{h,\bz_{h}}$ and $\Lambda_{h,\by_{h}}$ (cf. \eqref{eq:defhatGamF}), we have
\begin{equation*}
\Lambda_h[(\bsi_{\star,h}-\bsi_{\circ,h}, \bu_{\star,h}-\bu_{\circ,h}), (\bta_{h},\bv_{h})]\,=\,\mathcal{C}(\textbf{y}_{h},\bu_{\circ,h},\bta_{h})-\mathcal{C}(\textbf{z}_{h},\bu_{\star,h},\bta_{h})\,.
\end{equation*}
This result, combined with \eqref{eq:defGamF} by setting $(\bze_{h},\textbf{w}_{h}):=(\bsi_{\star,h}-\bsi_{\circ,h}, \bu_{\star,h}-\bu_{\circ,h})$, yields
\begin{equation*}
\begin{array}{c}
\Lambda_{h,\textbf{z}_{h}}[(\bsi_{\star,h}-\bsi_{\circ,h}, \bu_{\star,h}-\bu_{\circ,h}), (\bta_{h},\bv_{h})]\,=\,
\Lambda_h[(\bsi_{\star,h}-\bsi_{\circ,h}, \bu_{\star,h}-\bu_{\circ,h}), (\bta_{h},\bv_{h})]\\[2ex]
\,+\,\mathcal{C}(\textbf{z}_{h},\bu_{\star,h}-\bu_{\circ,h};\bta_{h})\\[2ex]
\,=\,\mathcal{C}(\textbf{y}_{h},\bu_{\circ,h};\bta_{h})-\mathcal{C}(\textbf{z}_{h},\bu_{\star,h};\bta_{h})+\mathcal{C}(\textbf{z}_{h},\bu_{\star,h}-\bu_{\circ,h};\bta_{h})\\[2ex]
\,=\,-\mathcal{C}(\bz_{h}-\by_{h},\bu_{\circ,h};\bta_{h})\,.
\end{array}
\end{equation*}
Now, we apply the discrete global inf--sup condition \eqref{QQ2} to the left-hand side of
the above equation and utilize the estimates \eqref{boundC} and \eqref{Stab}, to get
\begin{align*}
\dfrac{\alpha_{\Lambda,\mathtt{d}}}{2}\Vert(\bsi_{\star,h}-\bsi_{\circ,h}, \bu_{\star,h}-\bu_{\circ,h})\Vert_{h}&\,\leq\, \sup_{\textbf{0}\neq (\bta_{h},\bv_{h})\in \mathbb{X}_{h}\times \mathbf{Y}_{h}}\dfrac{\Lambda_{\textbf{z}_{h}}[(\bsi_{\star,h}-\bsi_{\circ,h}, \bu_{\star,h}-\bu_{\circ,h}), (\bta_{h},\bv_{h})]}{\Vert(\bta_{h},\bv_{h})\Vert_{h}}\nonumber\\[1mm]
&\,=\,\sup_{\textbf{0}\neq (\bta_{h},\bv_{h})\in \mathbb{X}_{h}\times \mathbf{Y}_{h}}\dfrac{\mathcal{C}(\textbf{y}_{h}-\textbf{z}_{h},\bu_{\circ,h};\bta_{h})}{\Vert(\bta_{h},\bv_{h})\Vert_{h}}\nonumber\\[1mm]
&\,\leq\, \dfrac{c_{\mathtt{em}}^{2}}{\nu}\,\Vert\textbf{z}_{h}-\textbf{y}_{h}\Vert_{1,h}\,\Vert\bu_{\circ,h}\Vert_{1,h}\nonumber\\[1mm]
&\,\leq\, \dfrac{2c_{\mathtt{em}}^{2}}{\nu\alpha_{\Lambda,\mathtt{d}}} \left(\Vert\textbf{f}\Vert_{0,3/4}+c_{g}\Vert\textbf{g}\Vert_{1/2,\partial\Omega} \right)\Vert\textbf{z}_{h}-\textbf{y}_{h}\Vert_{1,h}\,.
\end{align*}
So, thanks to the above inequality, we arrive at
\begin{align*}
\Vert\mathbf{S}_{\mathtt{d}}(\textbf{z}_{h})-\mathbf{S}_{\mathtt{d}}(\textbf{y}_{h})\Vert_{1,h}&\,=\,\Vert\bu_{\star,h}-\bu_{\circ,h}\Vert_{1,h}\\
&\,\leq\, \Vert(\bsi_{\star,h}-\bsi_{\circ,h}, \bu_{\star,h}-\bu_{\circ,h})\Vert_{h}\,\leq\, \dfrac{4}{\alpha_{\Lambda,\mathtt{d}}^2\nu} \left(\Vert\textbf{f}\Vert_{0,3/4}+c_g\Vert\textbf{g}\Vert_{1/2,\partial\Omega} \right)\Vert\textbf{z}_{h}-\textbf{y}_{h}\Vert_{1,h}\,,
\end{align*}
and thus ends the proof.
\end{proof}
The main result of this section is summarized in the following theorem.
\begin{Theorem}\label{WP_D}
Assume that the hypothesis of Lemma \ref{l_P2} hold.
Then, there exists $(\bsi_{h},\mathbf{u}_{h})\in \mathbb{X}_{0,h}\times \mathbf{Y}_{h}$ solution for Problem \ref{p:3}. Moreover, there holds 
\begin{equation}\label{StabG}
\Vert(\bsi_{h},\mathbf{u}_{h})\Vert_{h}\,\leq\, \dfrac{2}{\alpha_{\Lambda,\mathtt{d}}}\,
 \left(\Vert\mathbf{f}\Vert_{0,4/3}+c_{g}\Vert\mathbf{g}\Vert_{1/2,\partial\Omega} \right)\,.
\end{equation}
\end{Theorem}
\begin{proof}
The compactness of $\mathbf{S}_{\mathtt{d}}$ on space $\widehat{\mathbf{Y}}_{h}$ and the Lipschitz-continuity of $\mathbf{S}_{\mathtt{d}}$ are guaranteed by Lemmas \ref{l_P2} and \ref{l_P3}, respectivly. Hence, by applying the Banach fixed-point theorem directly to  Problem \ref{p:3}, one can conclude its existence and uniqueness. Furthermore, the stability result \eqref{StabG} is derived straightly from \eqref{Stab} stated in Lemma \ref{wp_S}.
\end{proof}
\section{Convergence Analysis}\label{sec5}
In this section, we are interested in deriving an a priori estimate for the error
\[
\Vert \bsi-\bsi_h\Vert_{\mathbb{H},h}\,,\quad \Vert\mathbf{u}-\mathbf{u}_h\Vert_{1,h}\,, \qan \Vert\mathbf{u}-\mathbf{u}_h\Vert_{0,\Omega}.
\] 
where $ (\bsi,\bu)\in \mathbb{X}\times\mathbf{Y} $ with $ \bu\in \widehat{\mathbf{Y}} $, is the unique solution of Problem \ref{p:2}, and $ (\bsi_{h},\bu_h)\in\mathbb{X}_{0,h}\times\mathbf{Y}_h $ with $ \bu_h\in \widehat{\mathbf{Y}}_h $ is the unique solution of Problem \ref{p:3}. As a byproduct of this, we also derive an a priori estimate for $ \Vert p -p_h\Vert_{0,\Omega} $, where $ p_h $ is the discrete pressure computed according to the postprocessing formula suggested by \eqref{As1a}, that is
\begin{equation}\label{qm}
	p_{h}\,:=\,-\dfrac{1}{2}\left(\operatorname{tr}(\bsi_{0h})+\operatorname{tr}(\mathbf{u}_{h}\otimes \mathbf{u}_{h})\right)-r_{\mathbf{u}_{h}}\,.
\end{equation}
To this end, we divide our result into suboptimal and optimal convergence for the velocity.
\subsection{Suboptimal convergence}
We begin by introducing the following lemmas, which will play
an essential role in the error analysis.
\begin{lemma}[trace inequality]
\label{l_tr}
Assume that the partition $\mathcal{K}_h$ satisfies the shape regularity assumptions $\mathbf{A1}$-$\mathbf{A4}$ stated in \cite{Wang14}. Then, there exists a constant $\mathcal{C}_{\mathtt{tr}}$ such that for any $K\in\mathcal{K}_h$ and edge/face $e\subset\partial K$, we have
\begin{equation}
\Vert v\Vert_{0,e}^{2}\,\leq\, \mathcal{C}_{\mathtt{tr}}\left(h_K^{-1}\Vert v\Vert_{0,K}^{2} +h_K \Vert \nabla v\Vert_{0,K}^{2}\right)\,.
\end{equation}
\end{lemma}
\begin{lemma}[inverse inequality]
\label{l_inv}
Let $K\subset\mathrm{R}^d$ be a d-simplex which has diameter $h_K$ and is shape regular. Assume that $p$ and $q$ be non-negative integers such that $p\leq q$. Then, there exists a
constant $\mathcal{C}_{\mathtt{inv}}$ such that
\begin{equation}
\Vert \varrho\Vert_{0,q,K}\,\leq\, \mathcal{C}_{\mathtt{inv}}\, h^{d(\frac{1}{q}-\frac{1}{p})}\,\Vert \varrho\Vert_{0,p,K}\,,
\end{equation}
for any polynomial $\varrho$ of degree no more than $k$.
\end{lemma}
\begin{lemma}\label{ls}
For $\bsi\in\mathbb{H}^{r+1}(\Omega)$ and $r\in N$ such that $r\leq k$, there exist positive constants $ C_0 $ and $C_{\mathtt{s}}$ such that
\begin{equation}\label{eql5:e1}
\sum\limits_{K\in \mathcal{K}_{h}}h_K\Vert \bsi\mathbf{n}-\Pcalbf_{b}^{K}(\bsi\mathbf{n})\Vert_{0,\partial K}^{2}\leq C_0 \, h^{2(r+1)}\Vert \bsi\Vert_{r+1}^{2}\,,
\end{equation}
and
\begin{equation}\label{eql5:e2}
\big|\mathcal{S}^{K}(\Pcalbb_{h}\bsi,\bta_{h})\big|\leq  C_{\mathtt{s}}h^{r+1}\Vert\bsi\Vert_{r+1}\Vert\bta_{h}\Vert_{\mathbb{H},h},
\end{equation}
for all $\bta_{h}\in \mathbb{X}_{h}$
\end{lemma}
\begin{proof}
The proof of vector versions of \eqref{eql5:e1} and \eqref{eql5:e2} can be found in \cite[eqs. (33) and (34)]{Gharibi21} and by following the similar arguments can be easily conclude them.
\end{proof}
Thanks to the projection error estimate stated in Lemma \ref{l1}, we only analyze the error functions defined by
\begin{equation*}
\btet_h := \Pcalbb_{h}\bsi-\bsi_{h}=\big\{  \Pcalbb_{0}^{K}\bsi - \bsi_{0h},  (\Pcalbf_{b}^{K}(\bsi \mathbf{n}_{e})-\bsi_{bh}\mathbf{n}_{e})\otimes \mathbf{n}_{e}\big\} \qan  \theta_h :=\boldsymbol{\mathcal{P}}_h \mathbf{u}-\mathbf{u}_h .
\end{equation*}
As the primary step in convergence analysis, we derive error equation as follows:
\begin{itemize}
\item[•]
 For any $\bta_{h}\in \mathbb{X}_{h}$, from the first equation of  Problem \ref{p:2}, we deduce\\
\begin{equation*}
	\begin{array}{c}
\mathcal{A}_{h}(\btet_h,~\bta_{h})+\mathcal{B}_{h}(\bta_{h},\theta_h)\,=\,\big[ \mathcal{A}_{h}(\Pcalbb_{h}\bsi,~\bta_{h})+\mathcal{B}_{h}(\bta_{h},\boldsymbol{\mathcal{P}}_{h}\mathbf{u}) \big]-\big[ \mathcal{A}_{h}(\bsi_{h},~\bta_{h})+\mathcal{B}_{h}(\bta_{h},\mathbf{u}_{h})\big]\\[2ex]
 \,=\,\mathcal{A}(\Pcalbb_{h}\bsi,~\bta_{h})+\mathcal{S}(\Pcalbb_{h}\bsi,~\bta_{h})+\mathcal{B}_{h}(\bta_{h},\boldsymbol{\mathcal{P}}_{h}\mathbf{u})+\mathcal{C}(\mathbf{u}_{h};\mathbf{u}_{h},\bta_{h})-\mathcal{G}(\bta_{h})\,.
 	\end{array}
\end{equation*}

We test the first equation of Problem \ref{p:1}, i.e., $\bsi^{\mathtt{d}}+(\mathbf{u}\otimes\mathbf{u})^{\mathtt{d}}-\nu\boldsymbol{\nabla} \mathbf{u}=\textbf{0}$ against $\bta_{0h}$, where $\bta_{h}\in \mathbb{X}_{h}$, and add it on the right-hand side of the above equation as well as use the L$^{2}$-orthogonality property of projections $\Pcalbb_{0}$, $\boldsymbol{\mathcal{P}}_{h}$, along with the definition of discrete weak divergence operator (cf. Definition \ref{d3}) and Green formula, to obtain  
\begin{equation*}
\begin{array}{c}
\mathcal{A}_{h}(\btet_h,~\bta_{h})+\mathcal{B}_{h}(\bta_{h},\theta_h)
\,=\,\dfrac{1}{\nu}\disp\int_{\Omega}(\Pcalbb_{0}\bsi)^{\mathtt{d}}:\bta_{0h}^{\mathtt{d}}+\mathcal{S}(\Pcalbb_{h}\bsi,~\bta_{h})+\mathcal{B}_{h}(\bta_{h},\boldsymbol{\mathcal{P}}_{h}\mathbf{u})+\mathcal{C}(\mathbf{u}_{h};\mathbf{u}_{h},\bta_{h})\\[2ex]
\,-\,\Big(\mathcal{G}(\bta_{h})+ \dfrac{1}{\nu}\disp\int_{\Omega}\bsi^{\mathtt{d}}:\bta_{0h}^{\mathtt{d}}-\disp\int_{\Omega}\boldsymbol{\nabla} \mathbf{u}:\bta_{0h}+\dfrac{1}{\nu}\disp\int_{\Omega}(\mathbf{u}\otimes\mathbf{u})^{\mathtt{d}}:\bta_{0h}\Big)\\[2.5ex]
\,=\,\mathcal{S}(\Pcalbb_{h}\bsi,~\bta_{h})+\Big(\mathcal{C}(\mathbf{u}_{h};\mathbf{u}_{h},\bta_{h})-\mathcal{C}(\mathbf{u};\mathbf{u},\bta_{h})\Big)\\[2ex]
\,+\,\disp\sum\limits_{K\in \mathcal{K}_{h}}\langle(\bta_{0h}-\bta_{bh})\mathbf{n},~\mathbf{u}-\boldsymbol{\mathcal{P}}_{h}\mathbf{u}\rangle_{\partial K}\,.
\end{array}
\end{equation*}
\item[•]
An application of the second equation of Problems \ref{p:1}, that is, $\operatorname{\textbf{div}} (\bsi)+\mathbf{f}=0$ with test function $\bv_{h}\in \mathbf{Y}_{h}$ and Definition \ref{d3} and Green formula, gives
\begin{equation}\label{eeq:pa}
	\mathcal{B}_{h}(\btet_h,\bv_{h})
	\,=\,\mathcal{B}_{h}(\Pcalbb_{h}\bsi,\bv_{h})-\mathcal{B}(\bsi,\bv_{h}) =\sum\limits_{K\in \mathcal{K}_{h}}\langle \Pcalbf_{b}^{K}(\bsi\mathbf{n})-\bsi\mathbf{n},\bv_{h}\rangle_{\partial K}\,.
\end{equation}
\end{itemize}
Hence, the error equation reads as follows.
\begin{prob}[error problem]
\label{p:5}
Find error functions $\btet_h\in\mathbb{X}_h$ and $\theta_h\in \mathbf{Y}_h$ such that
\begin{equation}\label{ErrEq}
\left\{\begin{array}{rcll}
\mathcal{A}_{h}(\btet_h,\bta_{h})+\mathcal{B}_{h}(\bta_{h},\theta_h)
+\big[\mathcal{C}(\mathbf{u};\mathbf{u},\bta_{h})-\mathcal{C}(\mathbf{u}_{h};\mathbf{u}_{h},\bta_{h})\big]&=&\mathcal{L}_1((\bsi,\mathbf{u});\bta_{h})\,,  \\[3mm]
\mathcal{B}_{h}(\btet_h,\mathbf{v}_{h})&=&\mathcal{L}_2(\bsi,\mathbf{v}_{h})\,, 
\end{array}\right.
\end{equation}
in which
\begin{align*}
\mathcal{L}_1((\bsi,\mathbf{u});\bta_{h})&:=\mathcal{S}(\Pcalbb_{h}\bsi,~\bta_{h})+\sum\limits_{K\in \mathcal{K}_{h}}\langle(\bta_{0h}-\bta_{bh})\mathbf{n},~\mathbf{u}-\boldsymbol{\mathcal{P}}_{h}\mathbf{u}\rangle_{\partial K}\,,\\
\mathcal{L}_2(\bsi,\mathbf{v}_{h})&:=\sum\limits_{K\in \mathcal{K}_{h}}\langle \Pcalbf_{b}^{K}(\bsi\mathbf{n})-\bsi\mathbf{n},\mathbf{v}_{h}\rangle_{\partial K}\,.
\end{align*}
for all $\bta_{h}\in \mathbb{X}_{h}$ and $\mathbf{v}_{h}\in \mathbf{Y}_{h}$.
\end{prob}
The following result states the estimate for error terms appeared in Problem \ref{p:5}.
\begin{lemma}\label{l_1}
Let $\bsi\in \mathbb{X}\cap\mathbb{H}^{r+1}(\Omega)$, $\mathbf{u}\in  \mathbf{Y}\cap\mathbf{H}^{r+2}(\Omega)$ and $ r\leq k $. Then, there exist the constants $ \mathcal{C}_1 $ and $ \mathcal{C}_2 $ such that
\begin{equation*}
	\begin{array}{rcll}
\big|\mathcal{L}_1((\bsi,\mathbf{u});\bta_{h})\big|
\,&\leq&\,  \mathcal{C}_{1}\, h^{r+1}\left(\|\bsi\|_{r+1}+\|\mathbf{u}\|_{r+2}\right)\| \bta_{h}\|_{\mathbb{H},h}&\quad\forall\, \bta_{h}\in \mathbb{X}_{h}\,,\\[2ex]
\big|\mathcal{L}_2(\bsi;\mathbf{v}_{h})\big|\,&\leq&\, \mathcal{C}_{2}\, h^{r+1}\|\bsi\|_{r+1}\|\mathbf{v}_{h}\|_{1,h}&\quad\forall\, \mathbf{v}_{h}\in \mathbf{Y}_{h}\,.
	\end{array}
\end{equation*}
\end{lemma}
\begin{proof}
See \cite[Lemma 4]{Gharibi21}.
\end{proof}
Now we are in position of establishing the main result of this section, namely, the suboptimal rate
of convergence for the weak Galerkin scheme provided by Problem \ref{p:3}.
\begin{Theorem}
\label{t_ERR}
 Assume that the data satisfy
\begin{equation}\label{AS_1}
\dfrac{4}{\nu\alpha_{\Lambda}\alpha_{\Lambda,\mathtt{d}}} \left(\Vert\textbf{f}\Vert_{0,4/3}+c_{g}\Vert\textbf{g}\Vert_{1/2,\partial\Omega}\right)\,\leq\, \dfrac{1}{2}\,,
\end{equation}
and let $(\bsi_{h},\bu_{h})\in\mathbb{X}_{h}\times \mathbf{Y}_{h}$ be the solution to Problem \ref{p:3} and $(\bsi,\bu)\in\mathbb{X}\times \mathbf{Y}$ be the solution of Problem \ref{p:2} satisfying the regularity conditions $\bsi\in\mathbb{H}^{k+1}(\Omega)$ and $\mathbf{u}\in \mathbf{H}^{k+2}(\Omega)$. Then, for all $k\in\mathrm{N}_{0}$, the following estimate holds
\begin{equation}\label{ER_SU}
\Vert\Pcalbb_{h}\bsi-\bsi_{h}\Vert_{\mathbb{H},h}+\Vert\boldsymbol{\mathcal{P}}_{h}\mathbf{u}-\mathbf{u}_{h}\Vert_{1,h}\,\leq\, \mathcal{C}_{\mathtt{sub}}\, h^{k+1}\left(\|\bsi\|_{k+1}+\|\mathbf{u}\|_{k+2}\right)\,.
\end{equation}
\end{Theorem}
\begin{proof}
Bearing in mind the definition of $ \Lambda_{h,\bu_{h}} $ (cf. \eqref{eq:defhatGamF}), utilizing Problem \ref{p:5} (cf. \eqref{ErrEq}), and some simple computations we obtain
\begin{equation}\label{eq1:T5}
	\begin{array}{c}
\Lambda_{h,\textbf{u}_{h}}[(\btet_{h},\theta_{h}), (\bta_{h},\bv_{h})]\,=\,\big[\mathcal{C}(\mathbf{u}_h;\theta_{h},\bta_{h})-\mathcal{C}(\mathbf{u};\mathbf{u},\bta_{h})+\mathcal{C}(\mathbf{u}_{h};\mathbf{u}_{h},\bta_{h})\big]\\[2ex]
\,+\,\mathcal{L}_1((\bsi,\mathbf{u});\bta_{h})+\mathcal{L}_2(\bsi,\mathbf{v}_{h})\,.
	\end{array}
\end{equation}
Next, adding and subtracting some suitable terms and recalling that $\theta_h=\boldsymbol{\mathcal{P}}_{h}\mathbf{u}-\mathbf{u}_{h}$ the first term on the right-hand side of the above equation can be rewritten as
\begin{equation}\label{eq2:T5}
	\begin{array}{c}
\mathcal{C}(\mathbf{u}_h;\theta_{h},\bta_{h})-\mathcal{C}(\mathbf{u};\mathbf{u},\bta_{h})+\mathcal{C}(\mathbf{u}_{h};\mathbf{u}_{h},\bta_{h})\,=\,\mathcal{C}(\mathbf{u}_h;\boldsymbol{\mathcal{P}}_{h}\mathbf{u}-\mathbf{u},\bta_{h})+\mathcal{C}(\mathbf{u}_h-\mathbf{u};\mathbf{u},\bta_{h})\\[2ex]
\,=\,\big[\mathcal{C}(\mathbf{u}_h;\boldsymbol{\mathcal{P}}_{h}\mathbf{u}-\mathbf{u},\bta_{h})+\mathcal{C}(\boldsymbol{\mathcal{P}}_{h}\mathbf{u}-\mathbf{u};\mathbf{u},\bta_{h}) \big]-\mathcal{C}(\theta_h;\textbf{u},\bta_{h})\,.
	\end{array}
\end{equation}
Now, it suffices to substitute \eqref{eq2:T5} back into \eqref{eq1:T5} and employ the discrete global inf-sup of $\Lambda_{h,\textbf{u}_{h}}$ (cf. Lemma \ref{l_iAC}, eq. \eqref{QQ2}), the continuity property of $\mathcal{C}$ (cf. \eqref{boundC}) and the estimats of $\mathcal{L}_1$ and $\mathcal{L}_2$ (cf. Lemma \ref{l_1}), to arrive at
\begin{equation}\label{eq3:T5}
\begin{array}{c}
\dfrac{\alpha_{\Lambda,\mathtt{d}}}{2}\Vert(\btet_{h},\theta_{h})\Vert_{h}\,\leq\, \dfrac{1}{\nu}\Vert \boldsymbol{\mathcal{P}}_{h}\mathbf{u}-\mathbf{u}\Vert_{\mathbf{Y}}\left(c_{\mathtt{em}}\Vert\mathbf{u}_h\Vert_{1,h}+\Vert\mathbf{u}\Vert_{\mathbf{Y}} \right)+\dfrac{c_{\mathtt{em}}}{\nu}\Vert\mathbf{u}\Vert_{\mathbf{Y}}\Vert\theta_h\Vert_{1,h}\\[2ex]
\,+\,\mathcal{C}_{3}\, h^{k+1}\left(\|\bsi\|_{k+1}+\|\mathbf{u}\|_{k+2}\right)\,,
\end{array}
\end{equation}
where $ \mathcal{C}_3 $ is a constant depending on $ \mathcal{C}_1 $, $ \mathcal{C}_2 $.

In this way, using the fact that $\mathbf{u}\in\widehat{\mathbf{Y}}$ and $\mathbf{u}_h\in\widehat{\mathbf{Y}}_h$, in particularly a priori bounds provided by \eqref{Stab} and \eqref{StabG}, and the inequality \eqref{eq3:T5} we deduce that
\begin{equation*}
\begin{array}{c}
\dfrac{\alpha_{\Lambda,\mathtt{d}}}{2}\Vert(\btet_{h},\theta_{h})\Vert_{h}\,\leq\, \left(\dfrac{4}{\nu\min\{\alpha_{\Lambda},\alpha_{\Lambda,\mathtt{d}}\}}\Vert \boldsymbol{\mathcal{P}}_{h}\mathbf{u}-\mathbf{u}\Vert_{\mathbf{Y}}+\dfrac{2}{\alpha_{\Lambda}\nu}\Vert(\btet_{h},\theta_{h})\Vert_{h}\right)
\left(\Vert\textbf{f}\Vert_{0,4/3}+c_{g}\Vert\textbf{g}\Vert_{1/2,\partial\Omega} \right)\\[2.4ex]
\,+\, \mathcal{C}_3 \, h^{k+1}\left(\|\bsi\|_{k+1}+\|\mathbf{u}\|_{k+2}\right)\,,
\end{array}
\end{equation*}

which together with assumption \eqref{AS_1} implies \eqref{ER_SU} with considering
\[
C_{\mathtt{sub}}\,:=\,\dfrac{2}{\alpha_{\Lambda,\mathtt{d}}}\max\bigg\{  \dfrac{4}{\nu\min\{\alpha_{\Lambda},\alpha_{\Lambda,\mathtt{d}}\}}\left(\Vert\textbf{f}\Vert_{0,4/3}+c_{g}\Vert\textbf{g}\Vert_{1/2,\partial\Omega} \right), \mathcal{C}_3\bigg\},
\]
and concludes the proof.
\end{proof}
\subsection{Optimal convergence}
To obtain an optimal order error estimate for the vector component $\theta_h=\mathbf{u}_h -\boldsymbol{\mathcal{P}}_h\textbf{u}$
in the usual $\mathbf{L}^2$-norm, we consider a dual problem that seeks $\boldsymbol{\vartheta}$ and $\Phi$ satisfying
\begin{equation}\label{eq:Stok}
\begin{cases} 
\hspace{.1cm}\boldsymbol{\vartheta}^{\mathtt{d}}+(\mathbf{u}\otimes\Phi)^{\mathtt{d}}\,=\,\nu\boldsymbol{\nabla} \Phi &  \qin \Omega\,, \\[1ex]
\hspace{1.0cm}-\operatorname{\mathbf{div}}( \boldsymbol{\vartheta})\,=\,\theta_h &  \qin \Omega\,, 
\\[1ex]
\disp\int_{\Omega}\tr(\boldsymbol{\vartheta})=0 \qan  \Phi\,=\,\mathbf{0}&  \qon \partial\Omega\,, 
\end{cases}
\end{equation}
Assume that the usual $\mathrm{H}^2$-regularity is satisfied for the dual problem; i.e., for any
$\theta_h\in\mathbf{L}^{2}(\Omega)$, there exists a unique solution $(\boldsymbol{\vartheta},\Phi)\in \mathbb{H}^{1}(\Omega)\times\mathbf{H}^{2}(\Omega)$ such that
\begin{equation}\label{ez:reg}
\Vert \boldsymbol{\vartheta}\Vert_{1}+\Vert\Phi\Vert_{2}\,\leq\, C_{\mathtt{reg}}\Vert \theta_h\Vert_{0,\Omega}\,.
\end{equation}
Next, we give an error estimate for $ \theta_{h} $ in the L$ ^2 $-norm.
\begin{Theorem}\label{t:opt}
Let the assumption of Theorem \ref{t_ERR} be satisfied. Also, assume that $(\boldsymbol{\vartheta}, \Phi) \in \mathbb{H}^{1}(\Omega)\times\mathbf{H}^{2}(\Omega)$ be a solution of \eqref{eq:Stok}.
Then, the following error estimation holds
\begin{equation}
\Vert \boldsymbol{\mathcal{P}}_h\mathbf{u}-\mathbf{u}_h\Vert_{0}\leq \mathcal{C}_{\mathtt{opt}}\, h^{k+2}\left(\Vert\mathbf{u}\Vert_{k+2}+\Vert\bsi\Vert_{k+1} \right)\,.
\end{equation}
\end{Theorem}
\begin{proof}
We divide the proof into three steps.\\[2mm]
\textbf{Step 1: discrete evolution equation for the error.} First, for any $\mathbf{v}\in\mathbf{H}^{1}(\Omega)$ using Definition \ref{d3} and Green formula, we obtain the following result
\begin{equation}\label{eq0:op}
\begin{array}{c}
\mathcal{B}_h(\bta_h, \boldsymbol{\mathcal{P}}_h\mathbf{v})\,=\,\disp\sum_{K\in\mathcal{K}_h}\big[-(\bta_{0h}, \boldsymbol{\nabla}(\boldsymbol{\mathcal{P}}_h\mathbf{v}))_K+\langle \bta_{bh}\mathbf{n}, \boldsymbol{\mathcal{P}}_h\mathbf{v} \rangle_{\partial K}\big]\\[2ex]
\,=\,\disp\sum_{K\in\mathcal{K}_h}\big[(\operatorname{\mathbf{div}}(\bta_{0h}), \boldsymbol{\mathcal{P}}_h\mathbf{v})_K+\langle (\bta_{bh}-\bta_{0h})\mathbf{n}, \boldsymbol{\mathcal{P}}_h\mathbf{v} \rangle_{\partial K}\big]\\[2ex]
\,=\,\disp\sum_{K\in\mathcal{K}_h}\big[(\operatorname{\mathbf{div}}(\bta_{0h}), \mathbf{v})_K+\langle (\bta_{bh}-\bta_{0h})\mathbf{n}, \boldsymbol{\mathcal{P}}_h\mathbf{v} \rangle_{\partial K}\big]\\[2ex]
\,=\,\disp\sum_{K\in\mathcal{K}_h}\big[-(\bta_{0h}, \boldsymbol{\nabla}\mathbf{v})_K+\langle\bta_{0h}\mathbf{n}, \mathbf{v}\rangle_{\partial K}+\langle (\bta_{bh}-\bta_{0h})\mathbf{n}, \boldsymbol{\mathcal{P}}_h\mathbf{v} \rangle_{\partial K}\big]\\[2ex]
\,=\,\disp\sum_{K\in\mathcal{K}_h}\big[-(\bta_{0h}, \boldsymbol{\nabla}\mathbf{v})_K+\langle (\bta_{0h}-\bta_{bh})\mathbf{n}, \mathbf{v}-\boldsymbol{\mathcal{P}}_h\mathbf{v} \rangle_{\partial K}+\langle\bta_{bh}\mathbf{n}, \mathbf{v}\rangle_{\partial K}\big]\,.
\end{array}
\end{equation}
Now, the testing of the Eq. \eqref{eq:Stok} against $(\btet_h, \theta_h)$ yields
\begin{subequations}
\begin{align}
\mathcal{A}(\btet_h,\boldsymbol{\vartheta})+\mathcal{C}(\mathbf{u};\Phi, \btet_h)&\,=\,\disp\int_{\Omega}\boldsymbol{\nabla}\Phi : \btet_{0h}\,,\label{eq1:op}\\[2ex]
-\mathcal{B}(\boldsymbol{\vartheta},\theta_h)&\,=\,(\theta_h,\theta_h)_{0,\Omega} \,.\label{eq2:op}
\end{align}
\end{subequations}
By employing \eqref{eq0:op} and the fact that $\Phi=\textbf{0}$ on $\partial\Omega$, the term on the right-hand side of \eqref{eq1:op} can be rewritten as:
\begin{equation*}
\disp\int_{\Omega}\boldsymbol{\nabla}\Phi : \btet_{0h}\,=\,-\mathcal{B}_h(\btet_h, \boldsymbol{\mathcal{P}}_h\Phi)+\sum_{K\in\mathcal{K}_h}\langle (\btet_{0h}-\btet_{bh})\mathbf{n}, \Phi-\boldsymbol{\mathcal{P}}_h\Phi \rangle_{\partial K}\,,
\end{equation*}
from which, replacing the first term on the right hand by the second row of \eqref{ErrEq}, and using the fact that $\Phi\in\mathbf{H}^1(\Omega)$ and $\Phi=\textbf{0}$ on $\partial\Omega$ gives
\begin{equation}\label{eq:nablaPhi}
	\begin{array}{c}
\disp\int_{\Omega}\boldsymbol{\nabla}\Phi : \btet_{0h}\,=\,\mathcal{L}_2(\bsi,\boldsymbol{\mathcal{P}}_h\Phi)+\disp\sum_{K\in\mathcal{K}_h}\langle (\btet_{0h}-\btet_{bh})\mathbf{n}, \Phi-\boldsymbol{\mathcal{P}}_h\Phi \rangle_{\partial K}\\[2ex]
\,=\,\underbrace{\sum\limits_{K\in \mathcal{K}_{h}}\langle \bsi\mathbf{n}-\Pcalbf_{b}^{K}(\bsi\mathbf{n}),\Phi-\boldsymbol{\mathcal{P}}_h\Phi\rangle_{\partial K}}_{=:\mathrm{I}_1}+\underbrace{\sum_{K\in\mathcal{K}_h}\langle (\btet_{0h}-\btet_{bh})\mathbf{n}, \Phi-\boldsymbol{\mathcal{P}}_h\Phi \rangle_{\partial K}}_{=:\mathrm{I}_2}\,.
	\end{array}
\end{equation}
On the other hand, by following the similar arguments of \eqref{eeq:pa} the term on the left-hand side of \eqref{eq2:op} can be rewritten by
\begin{align*}
\mathcal{B}(\boldsymbol{\vartheta},\theta_h)&=\mathcal{B}_h(\Pcalbb_h\boldsymbol{\vartheta},\theta_h)+\sum\limits_{K\in \mathcal{K}_{h}}\langle \boldsymbol{\vartheta}\mathbf{n}-\Pcalbf_{b}^{K}(\boldsymbol{\vartheta}\mathbf{n}),\theta_h\rangle_{\partial K}\\[1ex]
&=\mathcal{L}_1((\bsi,\mathbf{u});\Pcalbb_h\boldsymbol{\vartheta})-\mathcal{A}_{h}(\btet_h, \Pcalbb_h\boldsymbol{\vartheta})-\big[\mathcal{C}(\mathbf{u};\mathbf{u},\Pcalbb_h\boldsymbol{\vartheta})-\mathcal{C}(\mathbf{u}_{h};\mathbf{u}_{h},\Pcalbb_h\boldsymbol{\vartheta})\big]\nonumber\\[1ex]
&~~~+\sum\limits_{K\in \mathcal{K}_{h}}\langle \boldsymbol{\vartheta}\mathbf{n}-\Pcalbf_{b}^{K}(\boldsymbol{\vartheta}\mathbf{n}),\theta_h\rangle_{\partial K}\\
&=\mathcal{S}(\Pcalbb_{h}\bsi,~\Pcalbb_h\boldsymbol{\vartheta})-\mathcal{A}_{h}(\btet_h, \Pcalbb_h\boldsymbol{\vartheta})-\big[\mathcal{C}(\mathbf{u};\mathbf{u},\Pcalbb_h\boldsymbol{\vartheta})-\mathcal{C}(\mathbf{u}_{h};\mathbf{u}_{h},\Pcalbb_h\boldsymbol{\vartheta})\big]\\[1ex]
&~~~+\sum\limits_{K\in \mathcal{K}_{h}}\langle(\Pcalbb_{0}\boldsymbol{\vartheta}-\Pcalbb_{b}\boldsymbol{\vartheta})\mathbf{n},~\mathbf{u}-\boldsymbol{\mathcal{P}}_{h}\mathbf{u}\rangle_{\partial K}
+\sum\limits_{K\in \mathcal{K}_{h}}\langle \boldsymbol{\vartheta}\mathbf{n}-\Pcalbf_{b}^{K}(\boldsymbol{\vartheta}\mathbf{n}),\theta_h\rangle_{\partial K}\\
&=:\sum_{i=3}^{7} \mathrm{I}_i \,,
\end{align*}
where the first row of \eqref{ErrEq} and the definition of $\mathcal{L}_1$ were applied in the second and fourth lines, respectively.\\[2mm]
\textbf{Step 2: bounding the error terms $\mathrm{I}_1$-$\mathrm{I}_7$}. For the term $\mathrm{I}_1$, applying the Cauchy-Schwarz inequality, the bound \eqref{eql5:e1} stated in Lemma \ref{ls} and trace inequality (cf. Lemma \ref{l_tr}) we estimate
\begin{equation}\label{eq:I1}
	\begin{array}{c}
\big| \mathrm{I}_1 \big| \,\leq\, \left( \disp\sum\limits_{K\in \mathcal{K}_{h}}h_K\Vert \bsi\mathbf{n}-\Pcalbf_{b}^{K}(\bsi\mathbf{n})\Vert_{0,\partial K}^{2}\right)^{1/2}\left( \disp\sum\limits_{K\in \mathcal{K}_{h}}h_K^{-1}\Vert\Phi-\boldsymbol{\mathcal{P}}_h\Phi\Vert_{0,\partial K}^2\right)^{1/2}\\[2ex]
\,\leq\, Ch^{k+1}\Vert \bsi\Vert_{k+1}\left(\disp\sum\limits_{K\in \mathcal{K}_{h}}\big[h_K^{-2}\Vert\Phi-\boldsymbol{\mathcal{P}}_h\Phi\Vert_{0,K}^2 + \Vert\nabla(\Phi-\boldsymbol{\mathcal{P}}_h\Phi)\Vert_{0,K}^2 \big] \right)^{1/2}\\[2ex]
\,\leq\, Ch^{k+2}\Vert \bsi\Vert_{k+1}\Vert\Phi\Vert_{2}\,.
	\end{array}
\end{equation}
As for the term $\mathrm{I}_2$, the Cauchy-Schwarz and trace inequalities and the definition of discrete norm $\Vert\cdot\Vert_{\mathbb{H},h}$ (cf. first paragraph of Sec. \ref{sub.sec4.1}) implies that,
\begin{equation}\label{eq:I2}
	\begin{array}{c}
\big| \mathrm{I}_2 \big| \,\leq\, \left(\disp\sum_{K\in\mathcal{K}_h}h_K\Vert (\btet_{0h}-\btet_{bh})\mathbf{n}\Vert_{0,\partial K}^2 \right)^{1/2}\left(\disp\sum_{K\in\mathcal{K}_h}h_K^{-1}\Vert \Phi-\boldsymbol{\mathcal{P}}_h\Phi \Vert_{0,\partial K}^{2} \right)^{1/2}\\[2ex]
\,\leq\, \mathcal{C}_{\mathtt{tr}}\Vert \btet_h\Vert_{\mathbb{H},h}
\left(\disp\sum\limits_{K\in \mathcal{K}_{h}}\big[h_K^{-2}\Vert\Phi-\boldsymbol{\mathcal{P}}_h\Phi\Vert_{0,K}^2 + \Vert\nabla(\Phi-\boldsymbol{\mathcal{P}}_h\Phi)\Vert_{0,K}^2 \big] \right)^{1/2}\\[2ex]
\,\leq\, \mathcal{C}_{\mathtt{tr}}\, h \,\Vert \btet_h\Vert_{\mathbb{H},h}\Vert\Phi\Vert_{2}\,.
	\end{array}
\end{equation}
For term $\mathrm{I}_3$, we use the defintion of $\mathcal{S}(\cdot,\cdot)$ given by \eqref{def:Stab}, the continuity and orthogonality properties of operator $\Pcalbf_b^K$ and trace inequality, to get
\begin{equation}
	\begin{array}{c}
\big| \mathrm{I}_3\big|\,=\,\big| \mathcal{S}(\Pcalbb_{h}\bsi,~\Pcalbb_h\boldsymbol{\vartheta})\big|\\[2ex]
\,=\,\bigg|\disp\sum_{K\in\mathcal{K}_h} h_{K}\left\langle \Pcalbf_b^K((\Pcalbb_0^K\bsi)\mathbf{n}-\bsi\mathbf{n}),~ \Pcalbf_b^K((\Pcalbb_0^K\boldsymbol{\vartheta})\mathbf{n}-\boldsymbol{\vartheta}\mathbf{n})\right\rangle_{0,\partial K}\bigg|\\[2ex]
\,\leq\, \left( \disp\sum_{K\in\mathcal{K}_h} h_{K}\Vert \Pcalbf_b^K((\Pcalbb_0^K\bsi)\mathbf{n}-\bsi\mathbf{n})\Vert_{0,\partial K}^{2}\right)^{1/2}\left( \disp\sum_{K\in\mathcal{K}_h} h_{K}\Vert \Pcalbf_b^K((\Pcalbb_0^K\boldsymbol{\vartheta})\mathbf{n}-\boldsymbol{\vartheta}\mathbf{n})\Vert_{0,\partial K}^{2}\right)^{1/2}\\[2ex]
\,\leq\,\left( \disp\sum_{K\in\mathcal{K}_h} h_{K}\Vert (\Pcalbb_0^K\bsi)\mathbf{n}-\bsi\mathbf{n}\Vert_{0,\partial K}^{2}\right)^{1/2}\left( \disp\sum_{K\in\mathcal{K}_h} h_{K}\Vert (\Pcalbb_0^K\boldsymbol{\vartheta})\mathbf{n}-\boldsymbol{\vartheta}\mathbf{n}\Vert_{0,\partial K}^{2}\right)^{1/2}\\[2ex]
\,\leq\, \left(\disp\sum_{K\in\mathcal{K}_h} \Vert \Pcalbb_0^K\bsi-\bsi\Vert_{0, K}^{2}+h_{K}^2\Vert \boldsymbol{\nabla}(\Pcalbb_0^K\bsi-\bsi)\Vert_{0, K}^{2}  \right)^{1/2}\\[2ex]
\,\times\,\left(\disp\sum_{K\in\mathcal{K}_h} \Vert \Pcalbb_0^K\boldsymbol{\vartheta}-\boldsymbol{\vartheta}\Vert_{0, K}^{2}+h_{K}^2\Vert \boldsymbol{\nabla}(\Pcalbb_0^K\boldsymbol{\vartheta}-\boldsymbol{\vartheta})\Vert_{0, K}^{2}  \right)^{1/2}\\[2ex]
\,\leq\, Ch^{k+2}\Vert\bsi\Vert_{k+1}\Vert\boldsymbol{\vartheta}\Vert_{1}\,, 
\end{array}
\end{equation}
where the approximation property of the projector $\Pcalbb_0^K$ (cf. Lemma \ref{l1}) was used in the last step.
Next, in order to estimate $\mathrm{I}_4$, we add zero in the form $\pm \mathcal{A}_h(\btet_h, \boldsymbol{\vartheta})$, to find that
\begin{equation}\label{ez:0a}
\mathrm{I}_4 \,=\,\mathcal{A}_h(\btet_h, \Pcalbb_h^K\boldsymbol{\vartheta}_h- \boldsymbol{\vartheta})\,+\,\mathcal{A}_h(\btet_h, \boldsymbol{\vartheta})\,.
\end{equation}
The first term on the right-hand side of the above equation can be estimated using the continuity of $\mathcal{A}_h$ given in Lemma \ref{l_ah} and the approximation property of the projector $\Pcalbb_0^K$ as
\begin{equation}\label{ez:0b}
\big| \mathcal{A}_h(\btet_h, \Pcalbb_h^K\boldsymbol{\vartheta}_h- \boldsymbol{\vartheta})\big| \,\leq\, c_{\mathcal{A}}\, \Vert \btet_h\Vert_{\mathbb{H},h}\,\Vert \Pcalbb_h^K\boldsymbol{\vartheta}_h- \boldsymbol{\vartheta}\Vert_{\mathbb{H},h}\,\leq\, Ch\Vert \btet_h\Vert_{\mathbb{H},h}\Vert \boldsymbol{\vartheta}\Vert_{1}\,.
\end{equation}
Also, an application of the consistency property of $\mathcal{A}_h$ (cf. Lemma \ref{l1}) and testing the first row of \eqref{eq:Stok} with  $\btet_h$, yields
\begin{equation}\label{ez:0}
\mathcal{A}_h(\btet_h, \boldsymbol{\vartheta})\,=\,\mathcal{A}(\btet_h, \boldsymbol{\vartheta})\,=\,
-\disp\int_{\Omega}\boldsymbol{\nabla}\Phi: \btet_{0h}-\mathcal{C}(\textbf{u};\Phi, \btet_h)\,=\,
-\mathrm{I}_1-\mathrm{I}_2-\mathcal{C}(\textbf{u};\Phi, \btet_h)\,,
\end{equation}
where the the equivalent expression of $\int_{\Omega}\boldsymbol{\nabla}\Phi: \btet_{0h}$ given by \eqref{eq:nablaPhi} was used in the last step. Then, thanks to \eqref{eq:I1} and \eqref{eq:I2} it suffices to estimate the last term on the right-hand side. For this purpose, we rewrite the associated term by adding and subtracting some suitable term as
\begin{equation}\label{ez:1}
\mathcal{C}(\textbf{u};\Phi, \btet_h)\,=\,\mathcal{C}(\textbf{u};\Phi-\boldsymbol{\mathcal{P}}_h\Phi, \btet_h)\,+\,\mathcal{C}(\textbf{u};\boldsymbol{\mathcal{P}}_h\Phi, \btet_h)\,.
\end{equation}
To determine upper bounds for the right-hand side terms, we use H\"older inequality, the approximation and the continuity properties of the projector $\boldsymbol{\mathcal{P}}_h$, Gagliardo–Nirenberg inequality and inverse inequality. This gives
\begin{equation*}
\big| \mathcal{C}(\textbf{u};\Phi-\boldsymbol{\mathcal{P}}_h\Phi, \btet_h)\big|\,\leq\, \dfrac{1}{\nu}\Vert\textbf{u}\Vert_{\mathbf{Y}}\Vert \Phi-\boldsymbol{\mathcal{P}}_h\Phi\Vert_{\mathbf{Y}}\Vert\btet_h\Vert_{\mathbb{H},h}\,\leq\, \dfrac{1}{\nu}\Vert\textbf{u}\Vert_{\mathbf{Y}}\, h \,\Vert\Phi\Vert_{2}\,\Vert\btet_h\Vert_{\mathbb{H},h}\,,
\end{equation*}
and
\begin{align*}
\big| \mathcal{C}(\textbf{u};\boldsymbol{\mathcal{P}}_h\Phi, \btet_h)\big|&\,\leq\, \dfrac{1}{\nu}\Vert\textbf{u}\Vert_{\mathbf{Y}}\Vert \boldsymbol{\mathcal{P}}_h\Phi\Vert_{\mathbf{Y}}\Vert\btet_h\Vert_{\mathbb{H},h}\\[1mm]
&\,\leq\, \dfrac{1}{\nu}\Vert\textbf{u}\Vert_{\mathbf{Y}}\left(C_{\mathtt{GN}}\Vert \boldsymbol{\mathcal{P}}_h\Phi\Vert_{0}^{1/2}\Vert \nabla(\boldsymbol{\mathcal{P}}_h\Phi)\Vert_{0}^{1/2}\right)\Vert\btet_h\Vert_{\mathbb{H},h}\\[1mm]
&\,\leq\, \dfrac{1}{\nu}
\Vert\textbf{u}\Vert_{\mathbf{Y}}
\mathcal{C}_{\mathtt{inv}}h^{1/2}\left(C_{\mathtt{GN}}\Vert \boldsymbol{\mathcal{P}}_h\Phi\Vert_{0,4}^{1/2}\Vert \nabla(\boldsymbol{\mathcal{P}}_h\Phi)\Vert_{0,4}^{1/2}\right)\Vert\btet_h\Vert_{\mathbb{H},h}\\[1mm]
&\,\leq\, Ch\Vert\textbf{u}\Vert_{\mathbf{Y}}\Vert\Phi\Vert_{2}\Vert\btet_h\Vert_{\mathbb{H},h}\,.
\end{align*}
Combining the above bounds with \eqref{ez:1} and \eqref{ez:0}, along with \eqref{eq:I1}, \eqref{eq:I2} yields
\begin{equation*}
\big| \mathcal{A}_h(\btet_h, \boldsymbol{\vartheta})\big| \,\leq\, C\left(h^{k+2}\Vert\bsi\Vert_{k+1}+h\Vert\btet_h\Vert_{\mathbb{H},h}\right)\Vert\Phi\Vert_{2}\,,
\end{equation*}
which, together with estimate \eqref{ez:0b} gives
\begin{equation}
\big| \mathrm{I}_4\big|\,\leq\, C\left(h^{k+2}\Vert\bsi\Vert_{k+1}+h\Vert\btet_h\Vert_{\mathbb{H},h}\right)\Vert\Phi\Vert_{2}+Ch\Vert \btet_h\Vert_{\mathbb{H},h}\Vert \boldsymbol{\vartheta}\Vert_{1}\,.
\end{equation}
In what follow, we focus on the estimate of $\mathrm{I}_5$. To this end, using similar
arguments as in the proof of Theorem \ref{t_ERR} (cf. equation \eqref{eq2:T5}), we obtain
\begin{equation}
\begin{array}{c}
\mathrm{I}_5 \,=\,\mathcal{C}(\mathbf{u};\mathbf{u},\Pcalbb_h\boldsymbol{\vartheta})-\mathcal{C}(\mathbf{u}_{h};\mathbf{u}_{h},\Pcalbb_h\boldsymbol{\vartheta})\\[2ex]
\,=\,\mathcal{C}(\mathbf{u}-\boldsymbol{\mathcal{P}}_h\mathbf{u};\mathbf{u}, \Pcalbb_h\boldsymbol{\vartheta})+\mathcal{C}(\mathbf{u}_h;\mathbf{u}-\boldsymbol{\mathcal{P}}_h\mathbf{u}, \Pcalbb_h\boldsymbol{\vartheta})+\mathcal{C}(\theta_h;\mathbf{u},\Pcalbb_h\boldsymbol{\vartheta})+\mathcal{C}(\mathbf{u}_h;\theta_h,\Pcalbb_h\boldsymbol{\vartheta})\,,
\end{array}
\end{equation}
and therefore, each of the terms above are estimated as follows:
\begin{align*}
\big| \mathcal{C}(\mathbf{u}-\boldsymbol{\mathcal{P}}_h\mathbf{u};\mathbf{u}, \Pcalbb_h\boldsymbol{\vartheta})\big|&\,\leq\, \Vert \mathbf{u}-\boldsymbol{\mathcal{P}}_h\mathbf{u}\Vert_{0}\Vert \mathbf{u}\Vert_{0,4}\Vert \Pcalbb_0\boldsymbol{\vartheta}\Vert_{0,4}\leq Ch^{k+2}\Vert \mathbf{u}\Vert_{k+2}\Vert \mathbf{u}\Vert_{\mathbf{Y}}\Vert \boldsymbol{\vartheta}\Vert_{1}\\[1ex]
\big| \mathcal{C}(\mathbf{u}_h;\mathbf{u}-\boldsymbol{\mathcal{P}}_h\mathbf{u}, \Pcalbb_h\boldsymbol{\vartheta})\big|&\,\leq\, Ch^{k+2}\Vert \mathbf{u}\Vert_{k+2}\Vert \mathbf{u}_h\Vert_{\mathbf{Y}}\Vert \boldsymbol{\vartheta}\Vert_{1}\\[1ex]
\big| \mathcal{C}(\theta_h;\mathbf{u},\Pcalbb_h\boldsymbol{\vartheta})\big| &\,\leq\, \Vert\theta_h\Vert_{0,4}\Vert  \mathbf{u}\Vert_{0,4}\Vert\Pcalbb_h\boldsymbol{\vartheta}\Vert_0\\
&\,\leq\, \Vert\theta_{h}\Vert_{0,4}\Vert  \mathbf{u}\Vert_{\mathbf{Y}}\mathcal{C}_{\mathtt{inv}}h^{1/2}\Vert\Pcalbb_h\boldsymbol{\vartheta}\Vert_{0,4}\\[1ex]
&\,\leq\, (C_{\mathtt{GN}}\Vert\theta_{h}\Vert_{0}^{1/2}\Vert\nabla\theta_{h}\Vert_{0}^{1/2})\Vert  \mathbf{u}\Vert_{\mathbf{Y}}\mathcal{C}_{\mathtt{inv}}h^{1/2}\left(C_{\mathtt{GN}}\Vert\Pcalbb_h\boldsymbol{\vartheta}\Vert_{0}^{1/2}\Vert\nabla(\Pcalbb_h\boldsymbol{\vartheta})\Vert_{0}^{1/2}\right)\\[1ex]
&\,\leq\, h^{\frac{1}{4}}(C_{\mathtt{GN}}\Vert\theta_{h}\Vert_{0,4}^{1/2}\Vert\nabla\theta_{h}\Vert_{0}^{1/2})\Vert  \mathbf{u}\Vert_{\mathbf{Y}}\mathcal{C}_{\mathtt{inv}}^3 h^{\frac{3}{4}}\left(C_{\mathtt{GN}}\Vert\Pcalbb_h\boldsymbol{\vartheta}\Vert_{0,4}^{1/2}\Vert\nabla(\Pcalbb_h\boldsymbol{\vartheta})\Vert_{0}^{1/2}\right)\\[1ex]
&\,\leq\, Ch\Vert\nabla\theta_{h}\Vert_{0}\Vert  \mathbf{u}\Vert_{\mathbf{Y}}\Vert\nabla(\Pcalbb_h\boldsymbol{\vartheta})\Vert_{0}\leq Ch\Vert\theta_{h}\Vert_{h}\Vert  \mathbf{u}\Vert_{\mathbf{Y}}\Vert\boldsymbol{\vartheta}\Vert_{1}\\[1ex]
\big| \mathcal{C}(\mathbf{u}_h;\theta_h,\Pcalbb_h\boldsymbol{\vartheta})\big|&\,\leq\, Ch\Vert\theta_{h}\Vert_{h}\Vert  \mathbf{u}_h\Vert_{\mathbf{Y}}\Vert\boldsymbol{\vartheta}\Vert_{1}
\end{align*}
which leads to
\begin{equation}
\big| \mathrm{I}_5 \big|\,\leq\, C\left(h^{k+2}\Vert \mathbf{u}\Vert_{k+2}+h\Vert\theta_{h}\Vert_{h} \right)(\Vert  \mathbf{u}\Vert_{\mathbf{Y}}+\Vert  \mathbf{u}_h\Vert_{\mathbf{Y}})\Vert\boldsymbol{\vartheta}\Vert_{1}\,.
\end{equation}
On the other hand, by arguments similar to those used in the estimation term $\mathrm{I}_1$, we derive
\begin{equation*}
	\begin{array}{c}
\big| \mathrm{I}_6\big|\,=\,\disp\sum\limits_{K\in \mathcal{K}_{h}}\bigg| \langle(\Pcalbb_{0}\boldsymbol{\vartheta}-\Pcalbb_{b}\boldsymbol{\vartheta})\mathbf{n},~\mathbf{u}-\boldsymbol{\mathcal{P}}_{h}\mathbf{u}\rangle_{\partial K}\bigg|\\[1ex]
\,=\,\disp\sum\limits_{K\in \mathcal{K}_{h}}\bigg| \langle(\Pcalbb_{0}\boldsymbol{\vartheta}-\boldsymbol{\vartheta})\mathbf{n}-(\Pcalbb_{b}\boldsymbol{\vartheta}-\boldsymbol{\vartheta})\mathbf{n},~\mathbf{u}-\boldsymbol{\mathcal{P}}_{h}\mathbf{u}\rangle_{\partial K}\bigg|\\[1ex]
\,\leq\, \Vert \Pcalbb_{h}\boldsymbol{\vartheta}-\boldsymbol{\vartheta}\Vert_{\mathbb{H},h}\left(\disp\sum\limits_{K\in \mathcal{K}_{h}}\big[h_K^{-2}\Vert \textbf{u}-\boldsymbol{\mathcal{P}}_h \textbf{u}\Vert_{0,K}^2 + \Vert\nabla(\textbf{u}-\boldsymbol{\mathcal{P}}_h \textbf{u})\Vert_{0,K}^2 \big] \right)^{1/2}\\[1ex]
\,\leq\, Ch^{k+2}\Vert\textbf{u}\Vert_{k+2}\Vert\boldsymbol{\vartheta}\Vert_{1}\,.
\end{array}
\end{equation*}
Similarly, using \eqref{eql5:e1} we have
\begin{equation*}
	\begin{array}{c}
\big|\mathrm{I}_7\big| \,=\,\disp\sum\limits_{K\in \mathcal{K}_{h}}\bigg|\langle \boldsymbol{\vartheta}\mathbf{n}-\Pcalbf_{b}^{K}(\boldsymbol{\vartheta}\mathbf{n}),\theta_h\rangle_{\partial K}\bigg|
=\bigg|\langle \boldsymbol{\vartheta}\mathbf{n}-\Pcalbf_{b}^{K}(\boldsymbol{\vartheta}\mathbf{n}),\theta_h-\bar{\theta}_h\rangle_{\partial K}\bigg|\\[1ex]
\,\leq\, \left(\disp\sum\limits_{K\in \mathcal{K}_{h}}h_K\Vert \boldsymbol{\vartheta}\mathbf{n}-\Pcalbf_{b}^{K}(\boldsymbol{\vartheta}\mathbf{n})\Vert_{0,\partial K}^{2}\right)^{1/2}\left(\disp\sum\limits_{K\in \mathcal{K}_{h}}h_K^{-1}\Vert \theta_h-\bar{\theta}_h\Vert_{0,\partial K}^{2} \right)^{1/2}\\[1ex]
\,\leq\, Ch \Vert \boldsymbol{\vartheta}\Vert_1 \left(\disp\sum\limits_{K\in \mathcal{K}_{h}}\big[h_K^{-2}\Vert \theta_h-\bar{\theta}_h\Vert_{0,K}^{2}+\Vert\nabla\theta_h\Vert_{0,K}^2\big] \right)^{1/2}\\[2ex]
\,\leq\, Ch\Vert\theta_h\Vert_{1,h}\Vert \boldsymbol{\vartheta}\Vert_1 \,.
\end{array}
\end{equation*}
\textbf{Step 3: error estimate.} We now insert the bounds on $\mathrm{I}_1$-$\mathrm{I}_7$ in \eqref{eq2:op}, yielding
\begin{equation}
\Vert \theta_h\Vert_{0}\,\leq\, \Big(Ch^{k+2}+h(\Vert \theta_h\Vert_{1,h}+\Vert\btet\Vert_{\mathbb{H},h}) \Big)(\Vert \boldsymbol{\vartheta}\Vert_1+\Vert \Phi\Vert_2)\,.
\end{equation}
The sought result follows from Theorem \ref{t_ERR} and employing regularity \eqref{ez:reg}.
\end{proof}
We end this section by establishing the error estimate for the pressure.
By proceeding as in \cite [Theorem 5.5, eqs. (5.38) and (5.39)]{Gatica18}
(see also \cite[eq. 5.14]{Gatica21}), we deduce the existence of a positive constant $C$,
independent of $h$, such that
\begin{equation}\label{p-widehat-p}
	\|p - p_h\|_{0,\Omega} \,\le\, C\, \Big\{ \|\bsi - \bsi_h\|_{0,\Omega}
	\,+\, \|\bu - \bu_h\|_{0,4;\Omega} \Big\} \,.
\end{equation}
Therefore, thanks to Theorems \ref{t_ERR} and \ref{t:opt} we get
\begin{equation}\label{er_p}
	\Vert p-p_{h}\Vert_{0}\leq  C \, h^{k+1}\left(\|\bsi\|_{k+1}+\|\mathbf{u}\|_{k+2}\right).
\end{equation}
\newpage
\section{Numerical Results}\label{sec6}
The purpose of this section is to illustrate the efficiency of the WG pseudostress-based mixed-FEM for solving the Brinkman problem. 
Like in Ref. \cite{Gatica21}, we use a real Lagrange multiplier to impose the zero integral mean condition for $ \bsi_h $ of the discrete scheme. As a result, Poblem \ref{p:3} is rewritten as follow:
find $((\bsi_{h},\mathbf{u}_{h}),\lambda)\in \mathbb{X}_{0,h}\times\mathbf{Y}_{h}\times \mathrm{R}$ such that
\begin{equation}\label{NL}
\left\{\begin{array}{rcll}
\mathcal{A}_{h}(\bsi_{h}, \bta_{h})+\mathcal{C}(\mathbf{u}_{h},\mathbf{u}_{h};\bta_{h})+\mathcal{B}_{h}(\bta_{h},\mathbf{u}_{h})+\lambda\int_{\Omega}\operatorname{tr}(\bta_{h})&=&\mathcal{G}(\bta_{h}), &\quad \forall\, \bta_{h}\in \mathbb{X}_{h}\,,\\[3mm]
\mathcal{B}_{h}(\bsi_{h},\bv_{h})&=&\mathcal{F}(\bv_{h}), &\quad \forall \bv_{h}\in \mathbf{Y}_{h}\,,\\[3mm]
\xi\disp\int_{\Omega}\operatorname{tr}(\bsi_{h})&=&0, &\quad \forall\, \xi\in \mathrm{R}\,.
\end{array}\right.
\end{equation}
In addition, the individual errors associated to the main unknowns and the postprocessed pressure are 
denoted and defined, as usual, by
\[
\begin{array}{c}
	\disp
	{\tt e}(\boldsymbol{\sigma}) \,:=\, \Vert\bsi-\bsi_h
	\Vert_{\bdiv_{4/3};\Omega}\,, \quad 
	{\tt e}(\mathbf{u}) \,:=\, \Vert\mathbf{u}-\mathbf{u}_h\Vert_{0,4;\Omega}\,,  \qan
	{\tt e}(p) \,:=\, \Vert p-\widehat{p}_h\Vert_{0,\Omega}\,.
\end{array}
\]
In turn, for all $ \star\in \{\bsi,\bu,p\} $, 
we let ${\tt r}(\star)~:=~\dfrac{\text{log}({\tt e}(\star)/{\tt e}'(\star))}{\text{log}(h/h')}$ be the experimental 
rates of convergence, where $ h $ and $ h^{'} $ denote two consecutive mesh sizes with 
errors ${\tt e}(\star)$ and ${\tt e}^{'}(\star)$, respectively. 

In addition,
a fixed point strategy with a fixed tolerance $\mathtt{Tol}=1$e-6 is utilized for the solving of the nonlinear equation \eqref{NL}. To that end, we begin with vector all zeros as an initial guess and stop iterations when the velocity's error between two concessive iterations is adequately small,
\[
\Vert \bsi^{m}-\bsi_{h}^{m}\Vert_0 +
\Vert \mathbf{u}^{m}-\mathbf{u}_{h}^{m}\Vert_{\mathbf{Y}}\leq \mathtt{Tol}\,.
\]
In the first example, we mainly verify the accuracy of the constructed numerical scheme. Examples 2 and 3 is utilized to evaluate the effectiveness of the discrete scheme by simulation of practical problems for which no analytical solutions. We performed our computations using the MATLAB 2020 b software on an Intel Core i7 machine with 32 GB of memory.
\begin{figure}[t!]
	\begin{center}
		\includegraphics[width=.45\textwidth]{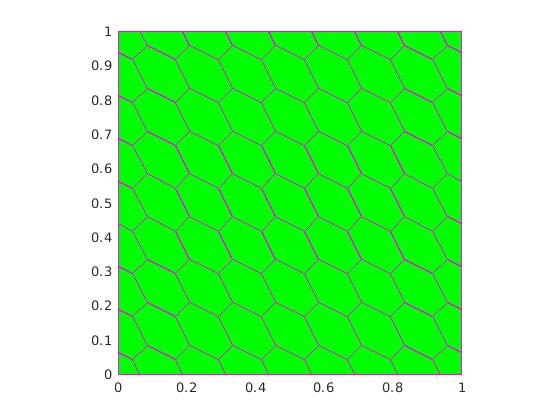} 
		\includegraphics[width=.45\textwidth]{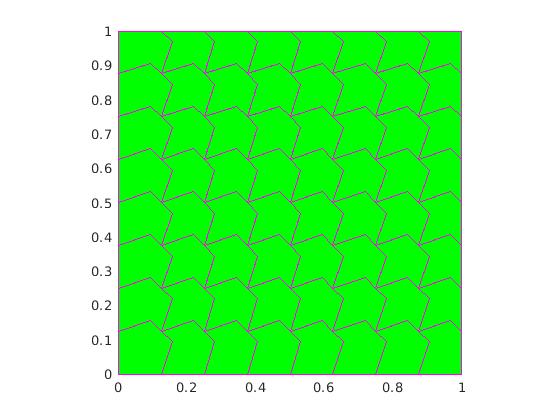} 
		\includegraphics[width=.45\textwidth]{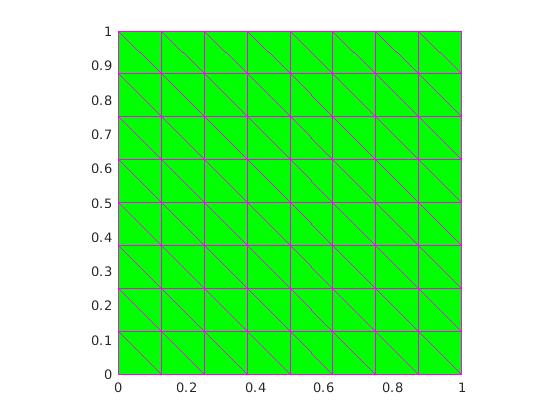}
	\end{center}
	\caption{Example 1, samples of the kind of meshes utilized.}\label{figure-1}
\end{figure}
\subsection{Example 1: Accuracy assessment}
We turn first to the numerical verification of the rates of convergence anticipated by Theorems \ref{t_ERR} and \ref{t:opt}. To this end, we consider parameter $\nu=0.1$ and  design the exact solution as follows:
\begin{equation*}
\begin{array}{l}
\mathbf{u}(x_1,x_2) \,=\,\begin{pmatrix} x_1^{2}\exp(-x_1)(1+x_2)\left(2\sin(1+x_2)+(1+x_2)\cos(1+x_2)\right)\\[2mm]
x_1(x_1-2)\exp(-x_1)(1+x_2)^{2}\sin(1+x_2) \end{pmatrix}\,,
\end{array}
\end{equation*}
\[
p(x_1,x_2) \,=\, \sin(2\pi x_1)\sin(2\pi x_2)\,,
\]
for all $ (x_1 , x_2)^{\mathtt{t}}\in \Omega =(0,1)^{2} $. The model problem is then complemented with the appropriate Dirichlet boundary condition.
Using the weak Galerkin spaces given in Sec. \ref{sec3} with polynomial
degree k$=0$, $1$, we solve Problem \ref{p:3} and obtain the approximated stress on a sequence of four succesively refined polygonal
meshes mades of hexagons, non-convex and triangular elements (see Fig. \ref{figure-1}). In addition, the discrete velocity field and the discrete pressure are computed using post-processing approach stated in Sec. \ref{sec5}. At each refinement level we compute errors between approximate and smooth exact solutions in L$^2$-norm. The results of this convergence study are collected in Tables \ref{Tab1}-\ref{Tab3}. One can see that the rate of convergence of individual stress and pressure variables is $ O(h^{k}) $, whereas it is $ O(h^{k+1}) $ for the velocity, which both are in agreement with the theoretical analysis stated in Theorems \ref{t_ERR} and \ref{t:opt}.
On the
other hand, in order to illustrate the accurateness of the discrete scheme, in Fig. \ref{fig1} we display components of the approximate velocity, stress, pressure on polygonal mesh with $h = 3.030$e-2 and k$=0$.
\begin{figure}[t!]
	\centering
	\begin{tabular}{ccc}
	\hspace{-2.2cm}\includegraphics[width=.44\textwidth]{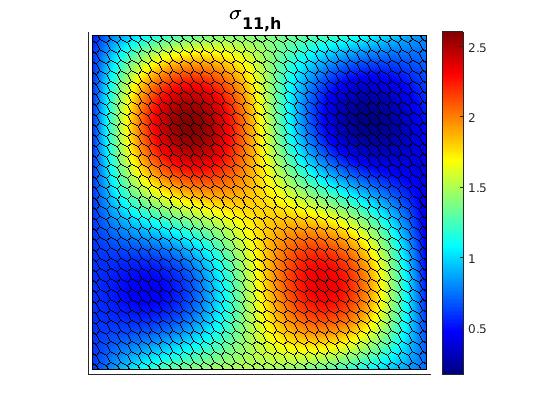}
\hspace{-1.33cm}	\includegraphics[width=.44\textwidth]{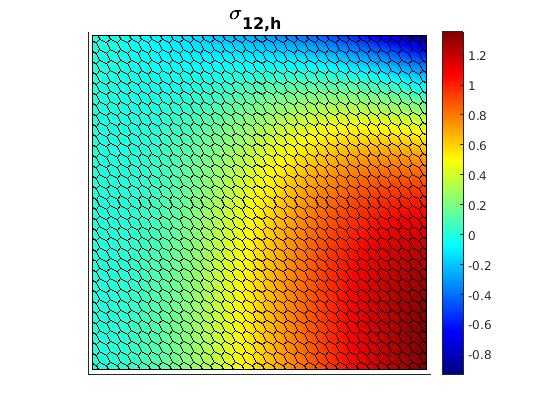}
	\hspace{-1.15cm}\includegraphics[width=.44\textwidth]{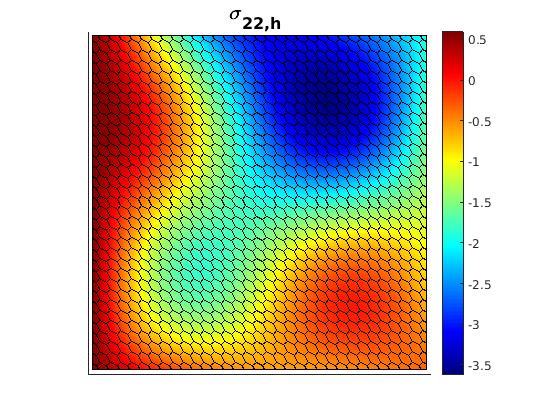}\\
	\hspace{-2.2cm}\includegraphics[width=.44\textwidth]{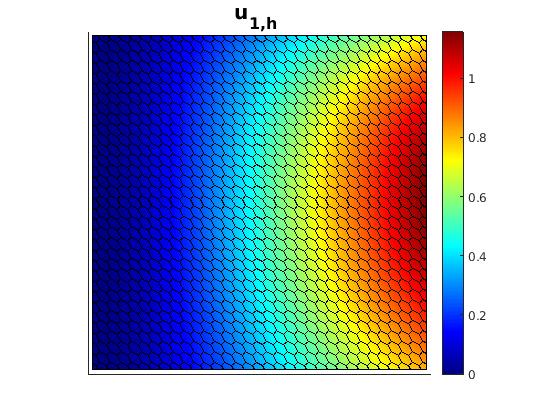}
\hspace{-1.33cm}	\includegraphics[width=.44\textwidth]{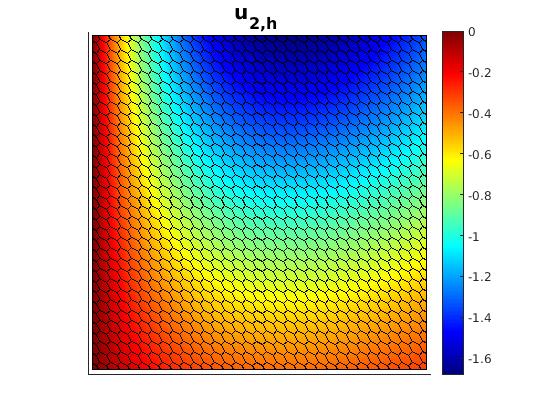}
\hspace{-1.15cm}\includegraphics[width=.44\textwidth]{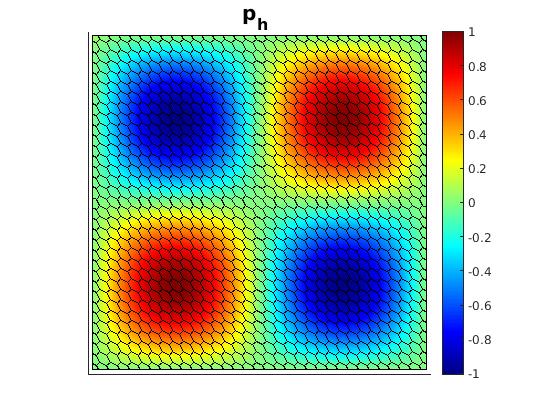}
	\end{tabular}
	\vspace{-.6cm}
	\caption{Example 1, snapshots of the numerical stress components 
		(first row, left to right), the velocity components and pressure (second row, left to right), computed with $k = 0$
		in the mesh made of hexagons with $h = 3.030$e-2. }\label{fig1}
\end{figure}

\begin{table}[t!]
	\caption{Example 1, history of convergence using hexagon}\label{Tab1}
	
	\bigskip
	
	\centerline{
		{\footnotesize\begin{tabular}{|cccccccc|}			
				\hline\hline
				$k$ & $h$ & $\mathtt{e}(\boldsymbol{\sigma})$ & ${\tt r}(\boldsymbol{\sigma})$ & 
				$\mathtt{e}(\mathbf{u})$ & ${\tt r}(\mathbf{u})$ & $\mathtt{e}(p)$ & ${\tt r}(p)$ \\
				\hline\hline
				{}& {2.000e-01}& {1.27585e+00}& {-}& {1.63990e-01}& {-}& {3.92748e-01}& {-}\\[1mm]
				{}& {1.111e-01}& {6.88337e-01}& {1.04985e+00}& {5.33247e-02}& {1.91124e+00}& {2.07409e-01}& {1.08623e+00}\\[1mm]
					
				{0}& {5.882e-02}& {3.14151e-01}& {1.23336e+00}& {1.58302e-02}& {1.90960e+00}& {9.47219e-02}& {1.23233e+00}\\[1mm]	
				{}& {3.030e-02}& {1.17182e-01}& {1.48674e+00}& {4.37012e-03}& {1.94051e+00}& {3.57120e-02}& {1.47063e+00}\\[1mm]
				\hline
				{}& {2.000e-01}& {3.34158e-01}& {-}& {3.63525e-02}& {-}& {9.11384e-02}& {-}\\[1mm]
				{}& {1.111e-01}& {7.65257e-02}& {2.50769e+00}& {5.47756e-03}& {3.21988e+00}& {2.16022e-02}& {2.44916e+00}\\[1mm]	
				{1}& {5.882e-02}& {1.64299e-02}& {2.41911e+00}& {7.67223e-04}& {3.09068e+00}& {5.13068e-03}& {2.26035e+00}\\[1mm]	
				{}& {3.030e-02}& {3.95366e-03}& {2.14755e+00}& {1.19847e-04}& {2.79901e+00}& {1.30515e-03}& {2.06382e+00}\\[1mm]	
				\hline
	\end{tabular}}}
\end{table}


\begin{table}[t!]
	\caption{Example 1, history of convergence using non-convex}\label{Tab2}
	
	\bigskip
	
	\centerline{
		{\footnotesize\begin{tabular}{|cccccccc|}			
				\hline\hline
				$k$ & $h$ & $\mathtt{e}(\boldsymbol{\sigma})$ & ${\tt r}(\boldsymbol{\sigma})$ & 
				$\mathtt{e}(\mathbf{u})$ & ${\tt r}(\mathbf{u})$ & $\mathtt{e}(p)$ & ${\tt r}(p) $ \\
				\hline\hline
				{}& {2.500e-01}& {1.92411e+00}& {-}& {2.16399e-01}& {-}& {5.24963e-01}& {-}\\[1mm]
				{}& {1.250e-01}& {9.61276e-01}& {1.00117e+00}& {6.67542e-02}& {1.69676e+00}& {2.69741e-01}& {9.60642e-01}\\[1mm]	
				{0}& {6.250e-02}& {3.68609e-01}& {1.38286e+00}& {1.91857e-02}& {1.79883e+00}& {1.04967e-01}& {1.36163e+00}\\[1mm]	
				{}& {3.125e-02}& {1.21234e-01}& {1.60429e+00}& {5.08183e-03}& {1.91661e+00}& {3.54124e-02}& {1.56761e+00}\\[1mm]
				\hline
				{}& {2.500e-01}& {6.42714e-01}& {-}& {4.33653e-02}& {-}& {1.82378e-01}& {-}\\[1mm]
				{}& {1.250e-01}& {2.05453e-01}& {1.64537e+00}& {6.90082e-03}& {2.65170e+00}& {6.23466e-02}& {1.54855e+00}\\[1mm]	
				{1}& {6.250e-02}& {5.59423e-02}& {1.87679e+00}& {1.35203e-03}& {2.35164e+00}& {1.76339e-02}& {1.82196e+00}\\[1mm]	
				{}& {3.125e-02}& {1.48719e-02}& {1.91135e+00}& {3.12928e-04}& {2.11122e+00}& {4.77578e-03}& {1.88455e+00}\\[1mm]
				\hline	
	\end{tabular}}}
\end{table}


\begin{table}[t!]
	\caption{Example 1, history of convergence using triangular}\label{Tab3}
	
	\bigskip
	
	\centerline{
		{\footnotesize\begin{tabular}{|cccccccc|}			
				\hline\hline
				$k$ & $h$ & $\mathtt{e}(\boldsymbol{\sigma})$ & ${\tt r}(\boldsymbol{\sigma})$ & 
				$\mathtt{e}(\mathbf{u})$ & ${\tt r}(\mathbf{u})$ & $\mathtt{e}(p)$ & ${\tt r}(p) $ \\
				\hline\hline
				{}& {1.768e-01}& {1.33737e+00}& {-}& {1.21833e-01}& {-}& {2.64186e-01}& {-}\\[1mm]
				{}& {8.839e-02}& {7.00927e-01}& {9.32063e-01}& {3.13360e-02}& {1.95902e+00}& {1.22818e-01}& {1.10503e+00}\\[1mm]	
				{0}& {4.419e-02}& {3.39792e-01}& {1.04461e+00}& {7.98954e-03}& {1.97164e+00}& {5.00518e-02}& {1.29503e+00}\\[1mm]	
				{}& {2.210e-02}& {1.63696e-01}& {1.05363e+00}& {2.01389e-03}& {1.98813e+00}& {1.99353e-02}& {1.32810e+00}\\[1mm]
				\hline
				{}& {1.768e-01}& {7.89707e-01}& {-}& {2.86149e-02}& {-}& {1.69212e-01}& {-}\\[1mm]
				{}& {8.839e-02}& {2.50022e-01}& {1.65926e+00}& {5.94594e-03}& {2.26679e+00}& {5.68132e-02}& {1.57454e+00}\\[1mm]	
				{1}& {4.419e-02}& {7.26955e-02}& {1.78212e+00}& {1.38897e-03}& {2.09789e+00}& {1.72263e-02}& {1.72162e+00}\\[1mm]	
				{}& {2.210e-02}& {2.02924e-02}& {1.84093e+00}& {3.40546e-04}& {2.02810e+00}& {4.95787e-03}& {1.79682e+00}\\[1mm]
				\hline	
	\end{tabular}}}
\end{table}
\subsection{Example 2 \cite{Hu19}: Lid-Driven Cavity Problem}
The next example is chosen to illustrate the performance of the
proposed method for modeling the lid-driven cavity flow in the square domain $\Omega=(0,1)^{2}$ with different values of $ \nu $.
For boundary conditions, we set the inflow $\textbf{g}=(1,0)^{\mathtt{t}}$ at the top end of $ \Omega $ and no-slip condition everywhere on the boundary.
In addition, the body force term is $\textbf{f}=\textbf{0}$.
In Fig. \ref{fig2} we display the computed velocity components and
pressure on hexagon mesh with $ h=3.03 $e-2 and $ k=0 $, which confirm the obtained results in \cite{Hu19}.
\begin{figure}[!h]
		\centering
	\begin{tabular}{ccc}
   \hspace{-1.3cm} \includegraphics[width=.4\textwidth]{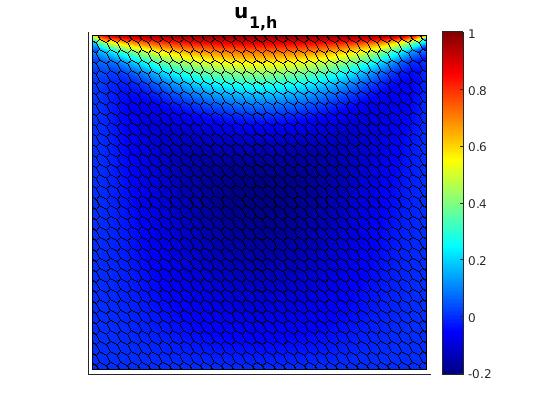}
   \hspace{-1.3cm}\includegraphics[width=.4\textwidth]{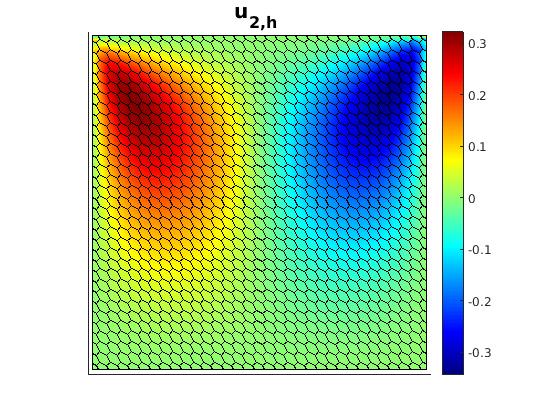}
   \hspace{-1.3cm}\includegraphics[width=.4\textwidth]{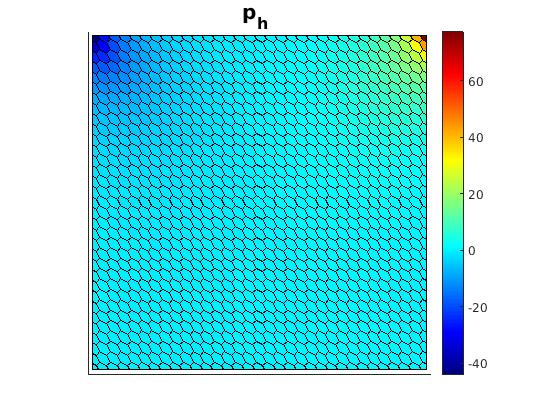}\\
  \hspace{-1.3cm} \includegraphics[width=.4\textwidth]{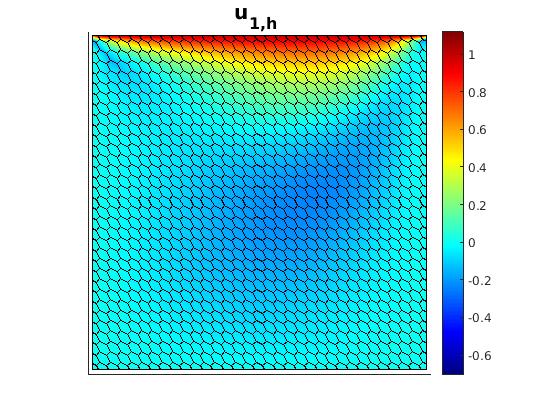}
   \hspace{-1.3cm}\includegraphics[width=.4\textwidth]{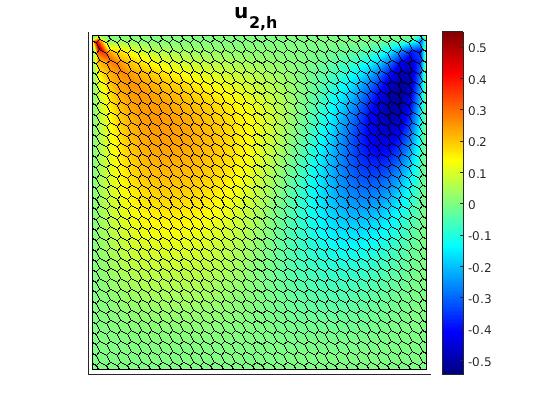}
   \hspace{-1.3cm}\includegraphics[width=.4\textwidth]{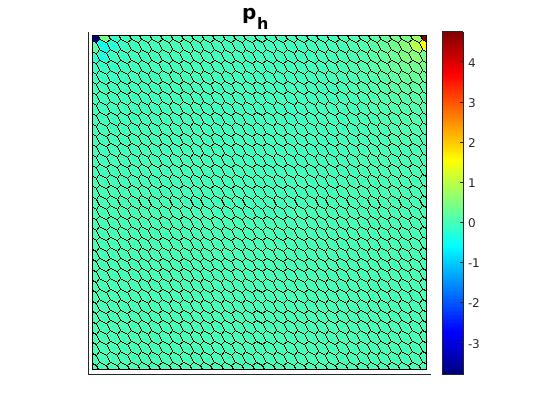}
    \end{tabular}
    \vspace{-.2cm}
	\caption{Example 2. The numerical velocity and pressure for $ \nu =1 $ (first row) and $ \nu=0.01 $ (second row) 
.}\label{fig2}
\end{figure}
\subsection{Example 4: Fluid flows with heterogeneous porous inclusions}
Mathematical modeling and simulation of fluid flows in the presence of single or multiple obstacles have been topics of interest for several decades due to their wide applicability in various practical circumstances across disciplines. Flow past solid bodies such as cylinders and airfoils has been investigated
broadly for a long time (see for instance, \cite{Fornberg,Li,Sohankar}) by using Navier-Stokes equations. 
We study the time-dependent Navier-Stokes equation, which is dependent on the porosity parameter, within the exterior flow domain ($\Omega_{f}$) characterized by large porosity and in the porous subdomains ($\Omega_{p}$) characterized by small porosity.
 A typical computational domain in the present problem is illustrated in Fig. \ref{fig7}-(a). Here, denoting domain by $\Omega=(-1,1)^{2}$ and the final time by $ t_F $, we consider the time-dependent Navier-Stokes equation with the porosity $\phi$ using the following non-dimensionalization
\begin{figure}[!h]
	\begin{center}
	\subfigure[]
    {\includegraphics[width=.45\textwidth]{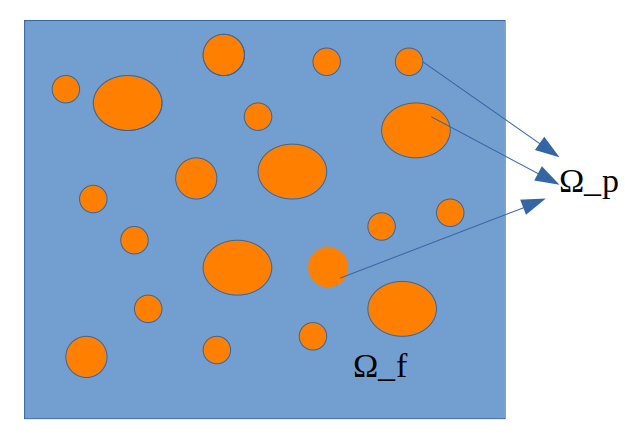}}
    \subfigure[]
    {\includegraphics[width=.31\textwidth]{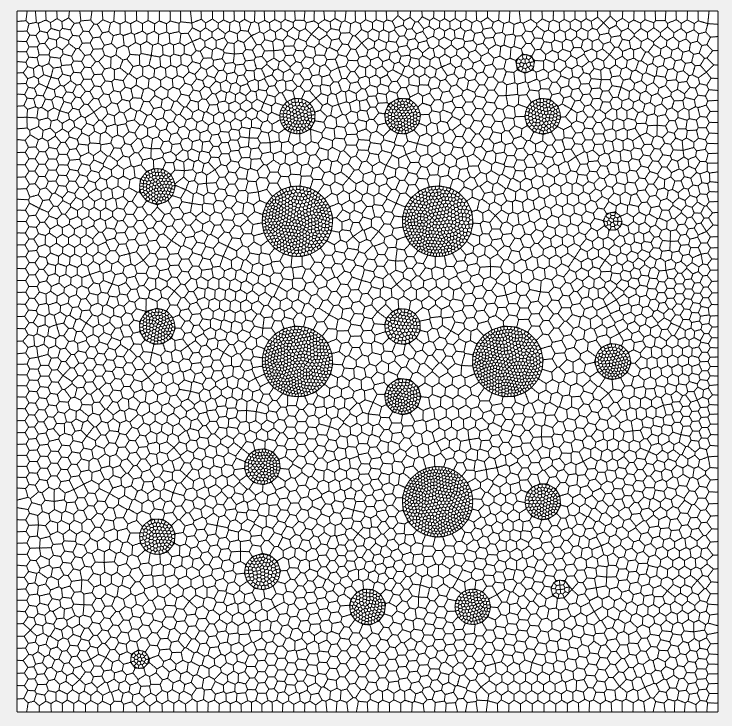}}
    \end{center}
    \vspace{-.2cm}
	\caption{Example 4. (a). Schematic of of the computational domain $\Omega=\Omega_{p}\cup\Omega_{f}$ , where $\Omega_{f}$ is the large porosity
subdomain and $\Omega_{p}$ is subdomain with the small porosity. (b). The polygonal mesh of domain.}\label{fig7}
\end{figure}
\begin{equation}
\begin{array}{rcll}
\dfrac{1}{\phi}\left(\mathbf{u}_{t}+\mathbf{u}\cdot\boldsymbol{\nabla} \mathbf{u}-\nu\boldsymbol{\Delta} \mathbf{u}\right)+\nabla p &=&\textbf{f} & \quad\qin\Omega\times(0,t_F]\,, \\[2ex]
\div(\boldsymbol{u})&=&0 & \quad \qin \Omega\times(0,t_F]\,, \\[2ex]
\bu(\cdot,0)&=& \mathbf{0} & \quad\qin\Omega\,.
\end{array}
\end{equation}
On the left boundary, we prescribe a inflow velocity $\mathbf{u}_{\mathtt{in}}=(0,1)^{\mathtt{t}}$. On
the right boundary,  we impose a zero normal Cauchy stress, which means that we need to set
\[
\big(\bsi+\bu\otimes\bu\big)\bn\,=\,\mathbf{0}\qquad\qon\Gamma_{out}\,,
\]
and on the remainder of the boundary we set no-slip velocity $ \bu=\mathbf{0} $. In addition,
 we performed discretization of time by Backward Euler method and taken 24 circular inclusions with different radii at relatively random locations. The performance of the proposed method has been tested with the following data:
\begin{equation*}
\begin{array}{llll}
\phi\,=\,\left\{\begin{array}{ll}
0.2& \qon\Omega_{p}\,,\\[1mm]
1&\qon\Omega_{f}\,,
\end{array}\right. & \quad \nu \,=\, 0.01\,, & t_{F}\,=\, 0.1\,, & \Delta t \,=\, 2e-3\,.
\end{array}
\end{equation*}
\begin{figure}[!h]
	\begin{center}
    \hspace{-.2cm}\includegraphics[width=.33\textwidth]{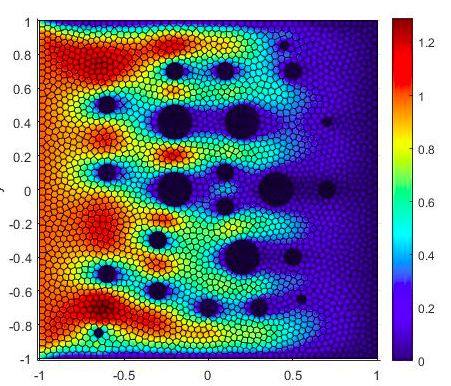}
   \includegraphics[width=.33\textwidth]{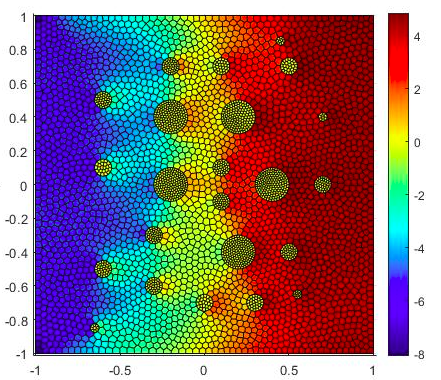}
   \includegraphics[width=.33\textwidth]{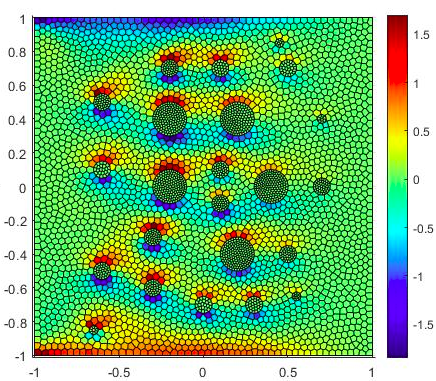}
    \end{center}
    \vspace{-.2cm}
	\caption{Example 4. The numerical velocity and the first two components of stress from left to right.}\label{fig8}
\end{figure}
In Fig. \ref{fig7}-(b), we depicted the polygonal mesh of the computational domain.
The results of the desired case simulation are presented in Fig. \ref{fig8}. It can be seen that the fluid is flowing faster in the area that has a large porosity; for the area that has a small porosity, the fluid is flowing slowly. In the area that has small porosity, we can see the gradation motion of the fluid clearly; this fact emphasizes that the proposed scheme can deal with the irregular pattern of porosity. 
\section{Declaration}
The authors declare that they have no known competing financial interests or personal relationships that could have appeared to influence the work reported in this paper.

\end{document}